\documentclass[12pt,a4paper]{amsart}
\usepackage{amssymb}
\usepackage{multirow}
\usepackage{hyperref}

\newcommand{\del}[1]{\frac{\partial}{\partial #1}}
\newcommand{\indel}[1]{\partial/\partial #1}

\newtheorem{theorem}{Theorem}
\newtheorem{lemma}[theorem]{Lemma}
\newtheorem{proposition}[theorem]{Proposition}
\newtheorem{corollary}[theorem]{Corollary}
\newtheorem*{mainthm}{Main Theorem}

\theoremstyle{remark}
\newtheorem{remark}[theorem]{Remark}
\newtheorem{example}[theorem]{Example}
\newtheorem{definition}[theorem]{Definition}

\title[Vector fields with single-valued solutions]{Meromorphic vector fields with single-valued solutions on complex surfaces}

\author{Adolfo Guillot}\address{Instituto de Matem\'aticas, Unidad Cuernavaca, Universidad Nacional Aut\'onoma de M\'exico,  A.P.~273-3 Admon.~3, Cuernavaca, Morelos, 62251, Mexico} \email{adolfo.guillot@im.unam.mx}

\thanks{The author thanks the \'Ecole Normale Sup\'erieure (Paris) for its hospitality during the  sabbatical leave where a first version of this work was finished. He thanks PASPA-DGAPA-UNAM (Mexico) for its support during this leave. The author was partially funded by the PAPIIT-UNAM IN108214 and  PAPIIT-UNAM IN102518 funds.}


\thanks{Published as: Adolfo Guillot, Meromorphic vector fields with single-valued solutions on complex surfaces, \emph{Adv. Math.} \textbf{354} (2019) 106742}

\begin{document}
\begin{abstract} We study ordinary differential equations in the complex domain given by meromorphic vector fields on compact complex K\"ahler  surfaces. We prove that if such an equation has a maximal single-valued solution (in particular, an entire one) whose image is not contained in a proper analytic subset then, up to a bimeromorphic transformation, either the vector field is holomorphic  or it preserves a fibration. 
\end{abstract}

\maketitle

\section{Introduction}

For algebraic differential equations in the complex domain, the existence of special kinds of solutions often imposes very restrictive conditions on the equations themselves. An instance of this is given by Malmquist's centennial result~\cite[Th\'eor\`eme~1]{malmquist}:
\begin{theorem}[Malmquist, 1913] Let~$R(w,t)$ be a rational function of~$w$ and~$t$ with complex coefficients. Let~$w$ be a transcendental meromorphic function defined in~$\mathbf{C}$ that is a solution to the differential equation
	$w'=R(w,t).$ Then 
	$R(t)=A(t)w^2+B(t)w+C(t)$, with~$A$, $B$ and~$C$ rational functions (the equation is a Riccati one).
\end{theorem}
Through the years, this theorem has known many proofs, from the early one by Yosida using Nevalinna theory~\cite{yosida} to the relatively recent  geometric one by Pan and Sebastiani~\cite{pan-sebastiani}. Generalizations of the theorem range from results valid for more general first-order nonautonomous equations (including some by Malmquist himself~\cite{malmquist2}) to some valid for higher order nonautonomous equations (although, in general, their conclusions are less categorical than that of Malmquist's original theorem). They go from fully algebraic to only partially algebraic ones. We refer the reader to~\cite[ch.~10]{laine} and~\cite{eremenko} for an overview of some of these generalizations.

Our main result is a generalization of Malmquist's theorem valid for differential equations given by meromorphic vector fields on compact complex K\"ahler surfaces (which are, in a way, autonomous, second-order, algebraic systems). Our result considers not only solutions given by meromorphic functions defined in the whole complex plane but, more generally, single-valued ones:

\begin{mainthm} Let~$S$ be a compact complex K\"ahler surface, $X$ a meromorphic vector field on~$S$, $P\subset S$ the curve of poles of~$X$, $\Omega\subset\mathbf{C}$ an open subset and~$\phi:\Omega\to S\setminus P$ a univalent maximal solution of the restriction of~$X$ to~$S\setminus P$ whose image is not contained in any proper analytic subset of~$S$. Up to a bimeromorphic transformation, either
	\begin{itemize}
		\item $X$ is holomorphic, or
		\item $S$ fibers over a rational or elliptic curve, with rational or elliptic fibers, and~$X$ preserves the fibration (in particular, each irreducible component of~$P$ is contained in a fiber).
	\end{itemize}
\end{mainthm}
In this statement, a solution~$\phi:\Omega\to S\setminus P$ is said to be \emph{univalent maximal}  if it has no analytic continuation as a map from~$\mathbf{C}$ into~$S\setminus P$ beyond~$\Omega$. This notion formalizes that of ``single valued solution'' and will be made precise in definition~\ref{maxsol}.  From the classification of holomorphic vector fields on compact K\"ahler surfaces (see proposition~\ref{vfkah} for a precise statement), it follows that if such a vector field does not preserve a fibration, it is a holomorphic vector field on a special Abelian surface (one without elliptic subgroups). In consequence, in the statement of the Main Theorem, we may replace the possibility ``$X$ is holomorphic'' by~``$X$ is a holomorphic vector field on an   Abelian surface without elliptic subgroups''. 

One can interpret the fibered case as the existence of a partial separation of variables: there is one variable that is integrated independently. For example, for the rational fibrations, we obtain, according to the nature of the vector field on the base, Riccati equations over rational, trigonometric,  or elliptic function fields.

Our Main Theorem generalizes Theorem~B of Rebelo and the author in~\cite{guillot-rebelo}, that states that the same conclusion may be obtained under the hypothesis that there is a univalent maximal solution through  every  point of~$S\setminus P$ (this is, that the vector field is \emph{semicomplete}). A related result is given by the tandem~\cite{brunella-complete}, \cite{bustinduy-giraldo}. In the first,
Brunella studied complete polynomial vector fields on~$\mathbf{C}^2$, giving a complete list of normal forms up to polynomial automorphisms, after proving that such a vector field must either have a first integral or preserve a fibration. In the second, Bustinduy and Giraldo, taking further Brunella's approach, prove that if a polynomial vector field on~$\mathbf{C}^2$ has one entire transcendental solution, it is actually complete. These results imply that if a polynomial vector field on~$\mathbf{C}^2$ has one entire transcendental solution, it preserves a fibration.

The structure of the surfaces and vector fields in the Main Theorem imply the existence of a threshold for the number of univalent maximal solutions that a vector field may have:
\begin{corollary}\label{treshold} Let~$X$ be a meromorphic vector field on a compact complex K\"ahler surface having five univalent maximal  solutions with different images such that the image of each one of them is not contained in a proper analytic subset. Then all the maximal solutions are univalent.
\end{corollary}
(The proof will be given in section~\ref{sec:excom}.) A vector field having exactly four univalent maximal solutions appears in example~\ref{four}. The Main Theorem also implies that, in our context, the domains where univalent maximal solutions are defined are very special:
\begin{corollary}Let~$S$ be a compact K\"ahler surface, $X$ a meromorphic vector field on~$S$. Let~$\Omega\subset\mathbf{C}$ and~$\phi:\Omega\to S$ a univalent maximal solution. Then~$\mathbf{C}\setminus\Omega$ is countable and, in particular, $\overline{\Omega}=\mathbf{C}$. 
\end{corollary}
Accordingly, it is not until dimension three that we may witness the existence of natural boundaries for functions satisfying algebraic differential equations (and we indeed do, like in the classical equations of Halphen, Chazy and Ramanujan~\cite{guillot-sl2}). These natural boundaries are inevitably accompanied by rich dynamics:
\begin{corollary}Let $S$ be a compact complex algebraic threefold, $X$ a meromorphic vector field on~$S$, $\phi:\Omega\to S$ a univalent maximal solution of~$X$ such that~$\mathbf{C}\setminus\Omega$ is uncountable (for instance, such that~$\overline{\Omega}\neq\mathbf{C}$). Then~$\phi(\Omega)$ is Zariski-dense.\end{corollary}

(Particular proofs of this result for Halphen's equations may be found in~\cite{macstr} and in~\cite[Thm.~A]{guillot-sl2}.) The last two corollaries are direct analogues of Corollaries~C and~D from~\cite{guillot-rebelo} in the present setting.

In the case where the ambient surface is algebraic, the proof of the Main Theorem is naturally split into two situations. One of them is when there is a maximal  and univalent solution~$\phi:\Omega\to S$ that cannot be extended to an entire mapping~$\widehat{\phi}:\mathbf{C}\to S$. In this situation, in section~\ref{sec:combi}  we will adapt (and at some points just refer to) the techniques of Rebelo and the author's article~\cite{guillot-rebelo}. The idea is to show that such a solution accumulates the locus of poles of~$X$ in a complicated way, and that this imposes severe restrictions on the nature of this locus. In the algebraic setting, these restrictions imply the existence of the fibration. In the other situation, the one where all the univalent maximal  solutions are (or can be extended to) entire maps~$\widehat{\phi}:\mathbf{C}\to S$, in section~\ref{sec:entire} we will use McQuillan's theorem on  holomorphic foliations on algebraic surfaces admitting an entire invariant curve (theorem~\ref{mq}) to obtain global information about the foliation, which  is then used together with some properties of the univalent maximal solutions to conclude. Despite the similarity of the conclusions of  McQuillan's theorem and of our main one, our result does not follow from McQuillan's  in a straightforward way (see example~\ref{ex:notriccati1}). McQuillan's result is an important ingredient in Brunella's classification of complete polynomial vector fields on~$\mathbf{C}^2$~\cite{brunella-complete}; it will play an analogous role in this work.

When the ambient surface is K\"ahler but not algebraic, we will deduce the Main Theorem from Brunella's classification of holomorphic foliations on these surfaces (section~\ref{sec:nonalg}).

We assume that the reader is familiar with general (mainly local) facts about foliations on surfaces, like those in~\cite[ch.~1]{brunella-birational}. For definitions related to the single-valuedness of solutions of complex differential equations, we refer the reader to~\cite[section~2]{guillot-rebelo} as well as to the next section.

\section{Preliminaries}
We briefly present some results related to the univalence of solutions of complex vector fields. A more detailed introduction may be found in sections~2, 3 and~4 of~\cite{guillot-rebelo}.

\subsection{Univalent maximal solutions} Let~$X$ be a holomorphic vector field on the complex manifold~$M$, let~$p\in M$. By Cauchy's theorem on the existence of solutions of ordinary differential equations, for every~$p\in M$ there exists an open set~$U\subset\mathbf{C}$, $0\in U$, and a holomorphic map~$\phi:(U,0)\to (M,p)$ that is a \emph{solution} of~$X$ in the sense that for every~$t\in U$, $\phi'(t)=X|_{\phi(t)}$. Solutions as these, defined in subsets of~$\mathbf{C}$, are said to be \emph{univalent}. By considering the maximal domain where the analytic continuation of~$\phi$ is defined (as a map from~$\mathbf{C}$ into~$M$) we obtain a function defined in some domain that spreads over~$\mathbf{C}$ but that is not, in general, a subset of~$\mathbf{C}$. (For example, the vector field~$e^z\indel{z}$ on~$\mathbf{C}$ has the multivalued solution~$t\mapsto -\log(1-t)$, defined in the Riemann surface of the logarithm.) These solutions are called \emph{maximal} and exist for each initial condition. In special settings, a maximal solution may also be univalent. The following definition formalizes the notion of ``single-valued solution'':
\begin{definition}\label{maxsol} Let~$X$ be a holomorphic vector field on the complex manifold~$M$. Let~$\Omega\subset\mathbf{C}$. A solution $\phi:\Omega\to M$ of~$X$ is said to be \emph{univalent maximal} if for every sequence~$\{t_i\}\subset\Omega$, converging in~$\mathbf{C}$ but not in~$\Omega$, the sequence~$\{\phi(t_i)\}$ escapes from every compact subset of~$M$ or, equivalently,  if the function~$\mathbf{I}\times \phi:\Omega\to\mathbf{C}\times M$ given by~$t\mapsto (t,\phi(t))$ is a proper one.  If~$X$ is a meromorphic vector field on the complex manifold~$M$, a \emph{univalent maximal solution} of~$X$ is a univalent maximal  solution of the restriction of~$X$ to the open subset where it is holomorphic.
\end{definition}
Thus, a solution is univalent if it is defined in some subset of~$\mathbf{C}$ and maximal if it has no analytic continuation beyond the domain where it is defined. If~$\phi:\mathbf{C}\to M$ is a solution of~$X$, it is automatically a univalent maximal one.   If all the maximal solutions are univalent the vector field is said to be \emph{semicomplete} (following Rebelo~\cite[def.~2.3]{rebelo-sing}) or simply \emph{univalent} (following Palais~\cite[def.~VI, p.~62]{palais}).

For a holomorphic vector field~$X$ on a curve, the \emph{time form} of~$X$ is the one-form~$\omega$ such that~$\omega(X)\equiv 1$. An orbit~$L$ of a holomorphic vector field may be parametrized by a univalent maximal  solution if every curve~$\gamma:[0,1]\to L$ for which $\int_\gamma \omega=0$ is closed~\cite[prop.~2.7]{rebelo-sing}.

\begin{lemma}\label{fewstat} If a vector field on a curve has~$n$ zeros (counted with multiplicity), its solutions have at least~$\lceil n/2\rceil$ determinations. If a meromorphic vector field on a curve admits a univalent maximal solution,  it has no poles (is in fact holomorphic) and has at most two zeros (counted with multiplicity).
\end{lemma}
\begin{proof} Let~$X$ be  a holomorphic vector field on a curve, $p$ a point  where it has a zero of multiplicity~$k+1$ ($k\geq 1$). In a suitable local coordinate, the vector field reads~$z^{k+1}/(1+\alpha z^k)\indel{z}$ for a unique~$\alpha\in\mathbf{C}$~\cite[thm~4.24]{ilyak}.  For the time form~$\omega$ of~$X$, $\alpha$ is the residue of~$\omega$ at~$p$. The inverse of the solution of~$X$ is,  a priori, multivalued, and given by
	\begin{equation}\label{integ-vf-w-zeros}f(z)=\int^{z}\omega =\int^z \left(\frac{\alpha}{w}+\frac{1}{w^{k+1}}\right)dw=\alpha\log(z)-\frac{1}{kz^k}+c.\end{equation}
	Suppose that~$\alpha\neq 0$ and let~$\eta(x)=e^{x/\alpha}$. The function~$\eta\circ f$ is no longer multivalued. It has an essential singularity at~$z=0$ and, by Picard's Great Theorem, its restriction to an arbitrarily small neighborhood~$U$ of~$0$ attains most values infinitely many times. Let~$q\in\mathbf{C}$ be a value attained infinitely many times under~$\eta\circ f$. The other elements of~$\eta^{-1}(q)$ are obtained by adding multiples of~$2i\pi\alpha$ to a given one, and are simultaneously attained by the same points that map to~$q$ under~$\eta\circ f$: for~$w$ in~$\eta^{-1}(q)$, the set~$f^{-1}(w)$ is independent of~$w$. Since~$q$ is attained infinitely many times under~$\eta\circ f$, every point in~$\eta^{-1}(q)$ is attained infinitely many times under~$f$. Thus, if~$\alpha\ne 0$, $f$ attains most values infinitely many times and its inverse, the solution of~$X$, has infinitely many determinations, and the lemma follows. Let us thus suppose that~$\alpha=0$. In this case, by~(\ref{integ-vf-w-zeros}), $f$ maps a neighborhood  of~$p$ to a neighborhood of the point at infinity in a $k$-fold way: the solution of~$X$, the inverse of~$f$, has~$k$ determinations. In the Riemann surface where it is defined, the solution of~$X$ maps a sector of angle~$2\pi k\geq \pi (k+1)$ at~$\infty$  to a neighborhood of~$p$. If~$X$ has multiplicity~$1$ at~$p$, it may be written as~$\lambda z\indel{z}$ in a suitable coordinate, and the solution is given by~$t\mapsto \exp(\lambda t)$: the preimage of any neighborhood of~$p$ contains, for some~$N$, the set~$\{t|\Re(\lambda t)<N\}$, containing a sector of angle~$\pi$ at infinity.  Thus, if~$\omega$ has vanishing residues at the multiple zeros of~$X$ then, in the Riemann surface of a maximal solution,  the sum of the angles of the sectors at~$\infty$ mapping to neighborhoods of the zeros of~$X$ is at least~$\pi n$. Under the projection of these sectors onto~$\mathbf{C}$, there is one point covered at least~$\lceil n/2\rceil$ times ($\lceil x\rceil$ denotes the \emph{ceiling} function of~$x$, the smallest integer greater or equal than~$x$). This proves the first part of the statement. 
	
	For the second part, notice that the first one implies that if a vector field admits a univalent maximal solution, it has at most two zeros (counted with multiplicity) and that, moreover, at any double zero the vector field is, in a suitable coordinate, given by~$z^2\indel{z}$ (compare with~\cite[section~3]{rebelo-sing}). Finally, the solutions of strictly meromorphic vector fields on curves cannot be single-valued:  a strictly meromorphic vector field having a pole of order~$n-1$ (for~$n>1$) may be written as~$\frac{1}{n}z^{1-n}\indel{z}$ (the time form is holomorphic, and its only invariant is the order of its zero). This vector field has the solution~$\sqrt[n]{t}$, which is multivalued (see also~\cite[lemma~2]{guillot-rebelo}). \end{proof}

\begin{remark}\label{rem:imp} Despite the fact that a vector field having a pole of order~$n-1$ has a solution with at least~$n$ determinations, the first part of the above lemma does not have an analogue involving the poles of a vector field. Let us give two examples in order to illustrate this. For the first, let~$\Sigma$ be a hyperelliptic curve of genus~$g$ and~$\pi:\Sigma\to\mathbf{P}^1$ the quotient by the hyperelliptic involution.  The pullback by~$\pi$ of a generic vector field with a double zero on~$\mathbf{P}^1$ is a meromorphic vector field on~$\Sigma$ having two double zeros and~$2(g+1)$ simple poles (at the ramification points), but whose solutions have only two determinations. For the second, consider the ramified double cover~$\Sigma$ of~$\mathbf{C}$ along~$\mathbf{Z}$ with the vector field induced by~$\indel{z}$. It has infinitely many poles, but its solutions have only two determinations.
\end{remark}

\begin{proposition}\label{vf-essential} If a meromorphic vector field on a curve has an isolated essential singularity, its solutions have infinitely many determinations.
\end{proposition}
\begin{proof} Consider the meromorphic vector field~$X=f(z)\indel{z}$  in a punctured neighborhood~$U^*$ of~$0$ in~$\mathbf{C}$ and suppose that its solutions have only finitely many determinations. By Lemma~\ref{fewstat}, the zeros of~$f$ do not accumulate to~$0$. The time form~$\omega=1/f(z)dz$ is holomorphic in~$U^*$. If the integral of~$\omega$ around~$0$ vanishes, consider the function~$h:U\to \mathbf{C}$ given by~$\int^z \omega$. If~$h$ has an essential singularity at~$0$ then, by Picard's Great Theorem, it attains most values infinitely many times: the solutions of~$X$ (given by the inverse of~$h$) have infinitely many determinations. We must conclude that~$h$ is meromorphic at~$0$, and thus that~$f$ is meromorphic as well. If the integral of~$\omega$ around~$0$ does not vanish, we may suppose, up to multiplying the vector field by a constant, that this integral is~$2i\pi$. Let~$h:U\to \mathbf{C}$,  $h(z)=\exp(\int^z \omega)$. Again, if~$h$ has an essential singularity at~$0$, by Picard's Great Theorem, it attains most values infinitely many times and the solutions of~$X$, given by~$t\mapsto h^{-1}(e^t)$, have, by an argument similar to the one used in the proof of lemma~\ref{fewstat}, infinitely many determinations.  Thus, $h$ is meromorphic and, since~$h'/h=1/f$, $f$ is meromorphic at~$0$.    
\end{proof}

In particular, a meromorphic vector field on~$\mathbf{C}$ whose solutions have finitely many determinations extends as a rational vector field to~$\mathbf{P}^1$.

\subsection{Affine structures on curves}

An \emph{affine structure} on a complex curve~$C$ is an atlas for its complex structure taking values in~$\mathbf{C}$ whose changes of coordinates lie within the affine group~$\mathrm{Aff}(\mathbf{C})=\{z\mapsto az+b\}$. An affine structure comes with a \emph{developing map}~$\mathcal{D}:\widetilde{C}\to\mathbf{C}$ and a \emph{monodromy} homomorphism~$\mathrm{mon}:\pi_1(C)\to \mathrm{Aff}(\mathbf{C})$ satisfying~$\mathcal{D}(\alpha\cdot p)=\mathrm{mon}(\alpha)(\mathcal{D}(p))$.

If the changes of coordinates of an affine structure lie within the group of translations~$\{z\mapsto z+b\}$, the affine structure is said to be a \emph{translation structure}. On a complex curve, translation structures are in correspondence with nowhere-vanishing holomorphic vector fields. If a curve~$C$ has a  nowhere-vanishing vector field~$X$ with time form~$\omega$, the charts of a translation structure are locally given by the primitives of~$\omega$ (the developing map of the translation structure induced by~$X$ is given by the integration of~$\omega$ along paths). Reciprocally, on a curve~$C$ endowed with a translation structure, the pull backs of the vector field~$\indel{z}$ on~$\mathbf{C}$ under the charts of the translation structure give a well-defined nowhere-vanishing vector field on~$C$.

An affine structure on a curve~$C$ is said to be \emph{uniformizable} if~$C$ is affinely equivalent to the quotient of an open subset of~$\mathbf{C}$ under the action of a group of affine transformations (with the affine structure inherited from the tautological affine structure of~$\mathbf{C}$).

In a curve endowed with a nowhere-vanishing holomorphic vector field, \emph{the induced affine structure is uniformizable if and only if the vector field admits a univalent maximal solution}. Let us sketch a proof of this fact. If~$C$ is a curve endowed with an affine structure, the only obstruction for this structure to be uniformizable is the existence of an open path~$\gamma:[0,1]\to C$, $\gamma(0)\neq \gamma(1)$, such that the developing map of the affine structure along the image of~$\gamma$ maps the endpoints of~$\gamma$ to the same point. If~$C$ has a nowhere-vanishing vector field~$X$ with time form~$\omega$ and~$X$ does not have a univalent maximal solution, there exists an open path~$\gamma:[0,1]\to C$ such that~$\int_\gamma \omega=0$, which, by the previous arguments, prevents the affine structure from being uniformizable.

Given two affine structures on a disk~$\Delta$ with corresponding coordinate charts~$g_i:\Delta\to \mathbf{C}$, $i=1,2$, for~$h=g_2\circ g_1^{-1}$, the one-form given by the pullback of\begin{equation}\label{affinedefect}\frac{h''(z)}{h'(z)}dz\end{equation}
by~$g_1$ depends only on the affine structures (and not on the particular coordinate charts). It vanishes if and only if~$g_1$ and~$g_2$ define the same affine structure, thus measuring the difference of the affine structures. Reciprocally, given a holomorphic one-form~$\eta$ and an affine structure on~$\Delta$, there is a second affine structure on~$\Delta$ such that the difference with the original one is~$\eta$.

It will be important to understand the affine structures defined in the complement of a discrete set of points (affine structures with \emph{singularities}) and, among them, the \emph{uniformizable} ones (as before, those that,  in the complement of the singularities, are affinely equivalent to the quotient of an open subset of~$\mathbf{C}$ by a group of affine transformations). Consider an affine structure on~$\Delta^*=\Delta\setminus\{0\}$ that does not extend as an affine structure to~$0$ (an affine structure with a \emph{singularity} at~$0$) and an auxiliary affine structure on~$\Delta$. Their difference is a one-form~$\eta$ in~$\Delta^*$ that is not holomorphic at~$0$. The \emph{ramification index} at~$0$ of the affine structure with singularities is
\[\mathrm{ind}(\Delta,0)=\frac{1}{\mathrm{Res}(\eta,0)+1}.\]
For the affine structure on~$\Delta$ whose charts are branches of~$\sqrt[n]{z}$, $n\in\mathbf{Z}^*$, the ramification index is~$n$; for the one whose developing map is~$\log(z)$, it is~$\infty$. In all these cases the difference with a regular affine structure on~$\Delta$ has a simple pole. All these local affine structures with singularities are uniformizable and are in fact the only ones. Moreover, they are characterized by the invariants of the one-form: \emph{an affine structure on~$\Delta^*$ with a singularity at~$0$ is uniformizable if and only if for the one-form~$\eta$ measuring its difference with a regular affine structure on~$\Delta$, $\eta$ has a simple pole at~$0$ with residue of the form~$1/n-1$, $n\in\mathbf{Z}^*\cup\{\infty\}$} (see~\cite[prop.~6]{guillot-rebelo}). 

This local result globalizes on compact curves as follows (see~\cite[prop.~7]{guillot-rebelo}):

\begin{proposition} \label{affinecomclas} Up to affine equivalence, the uniformizable affine structures with singularities on compact curves are:
	\begin{itemize} 
		\item Rational orbifolds. The rational curve $\mathbf{C}\cup\{\infty\}$ (with the tautological affine structure of~$\mathbf{C}$) and the quotients of the latter by the cyclic groups of linear transformations of order~$n$ fixing~$0$ and~$\infty$, realized by~$z\mapsto z^n$. The ramification indices are~$n$ (at~$0$) and~$-n$ (at~$\infty$).
		\item Parabolic. (i) The parabolic cylinder, $\mathbf{C}/2i\pi\mathbf{Z}$ compactified by~$\mathbf{P}^1$ via the map~$w= e^z$;  equivalently, $\mathbf{P}^1$ with the affine structure induced by the vector field~$w\indel{w}$, with two points of ramification index~$\infty$, at~$0$ and~$\infty$. (ii) The orbifold~$(2,2,\infty)$, quotient of the latter by the involution~$w\mapsto 1/w$ (induced by~$z\mapsto -z$), uniformized by~$\cos(iz)=\frac{1}{2}(e^z+e^{-z})$, having one point with ramification index~$\infty$ and two with ramification index~$2$. 
		\item Elliptic. The curves and orbifolds arising as compact quotients of~$\mathbf{C}$ under the action of (crystallographic) subgroups of the affine group. These groups contain a lattice as a normal subgroup of finite index and the curves are thus quotients of elliptic curves. As curves, they are either elliptic (without singularities for the affine structure), or rational curves with singularities of ramification indices~$p_i$ with~$\sum (1-1/p_i)=2$. The possibilities are~$(2,3,6)$, $(2,4,4)$, $(3,3,3)$ or~$(2,2,2,2)$, the last one constituting a family parametrized by the cross-ratio.
		\item Hopf tori. Elliptic curves of the form $\mathbf{C}^*/G$, $G$  a discrete subgroup of~$\mathbf{C}^*$ containing an element~$\lambda$ with~$|\lambda|\neq 1$ (the affine structure has no singularities).
	\end{itemize}
\end{proposition}
The result is, of course, related to the classification of the crystallographic groups of the plane. We refer to~\cite[prop.~7]{guillot-rebelo} for a proof along the lines of this article (see also~\cite{bb} for an early appearance of this classification in the setting of complex differential equations).

\begin{remark}\label{orbifold} It may be useful to consider the curves in the above list as orbifolds (declaring that at a point where the affine structure has ramification index~$n\neq \infty$, the local angle is~$2\pi/n$, and removing the points where the affine structure has ramification index~$\infty$). Under such convention, the monodromy representation associated to the affine structure is naturally a representation of the orbifold fundamental group. \end{remark}

\subsection{Foliations and vector fields on surfaces} We refer to~\cite[ch.~1]{brunella-birational} for a swift presentation of some general facts about the local theory of foliations and vector fields on surfaces, to~\cite[section~5.3]{loray} for a more detailed one.

Let~$X$ be a meromorphic vector field on the complex surface~$S$ and let~$\mathcal{F}$ be the foliation with singularities induced by~$X$. In the open subset where~$X$ is holomorphic and nonzero, the leaves of the induced foliation~$\mathcal{F}$ are naturally endowed with a translation structure (in particular, an affine one) whose transverse variation is holomorphic. If a curve~$C$ is an irreducible component of the locus of zeros or poles of~$X$ that is invariant by~$\mathcal{F}$, the leafwise affine structure induced by~$X$ extends to~$C$~\cite[prop.~8]{guillot-rebelo}, although the translation structure does not.  For example, for~$X=f(x,y)y^q\indel{x}$ with~$q\in\mathbf{Z}\setminus\{0\}$, $f$ holomorphic and nonzero, on the leaf~$\{y=0\}$, along which~$X$ has zeros or poles, this affine structure is the one induced by the vector field~$y^{-q}X$ (on the leaf~$y=y_0$, $y_0\neq 0$, the vector fields~$X$ and~$y^{-q}X$ are proportional and induce thus the same affine structure). In this way, in the complement of the singularities of~$\mathcal{F}$ and of the curves of zeros and poles that are not invariant by~$\mathcal{F}$, the foliation~$\mathcal{F}$ admits a leafwise affine structure. It extends as an affine structure  with singularities  to the nonsingular points of~$\mathcal{F}$ lying at the curve of zeros and poles.

\begin{definition} Let~$X$ be a meromorphic vector field on the complex surface~$S$, $\mathcal{F}$ the foliation induced by~$X$. The~\emph{uniformizable} or \emph{univalent} locus, denoted by $\mathfrak{U}$, is the subset of~$S\setminus\mathrm{Sing}(\mathcal{F})$ of leaves carrying a uniformizable affine structure (with singularities).
\end{definition}
The set~$\mathfrak{U}$ is naturally saturated by~$\mathcal{F}$. All the leaves of~$\mathcal{F}$ parametrized by a univalent maximal solution belong to~$\mathfrak{U}$.

\begin{proposition}\label{unifisclosed} Let~$X$ be a meromorphic vector field on the complex surface~$S$. The uniformizable locus~$\mathfrak{U}\subseteq S$ associated to~$X$  is closed. \end{proposition} 
\begin{proof}[Sketch of proof]  Lack of uniformizability of an affine structure on a curve~$C$ is equivalent to the existence of an open path~$\gamma:[0,1]\to C$, $\gamma(0)\neq\gamma(1)$ such that the developing map of the affine structure along the image of~$\gamma$ maps~$\gamma(0)$ and~$\gamma(1)$ to the same point. If~$C$ is a leaf of~$\mathcal{F}$ where the affine structure is not uniformizable and~$\gamma$ is as above, we can lift~$\gamma$ to a path~$\gamma'$ in a neighboring leaf~$C'$ in such a way that the lift extends to a local isomorphism of the affine structures along neighborhoods of the images of~$\gamma$ and~$\gamma'$. But this means that the affine structure on~$C'$ is not uniformizable (see~\cite[cor.~12]{guillot-rebelo} for details).\end{proof}

Within~$\mathfrak{U}$, the leafwise geometry conditions the holonomy. This is the content of the \emph{Fundamental Lemma}, for the proof of which we refer the reader to~\cite[section~4.2]{guillot-rebelo}:

\begin{lemma}[Fundamental Lemma]\label{holmon} Let~$X$ be a holomorphic vector field on the surface~$M$, $\mathcal{F}$ the foliation induced by~$X$. Let~$\mathfrak{U}\subset M\setminus \mathrm{Sing}(\mathcal{F})$ denote the set given by the leaves where the induced affine structure (with singularities) is uniformizable. Let~$p\in\mathfrak{U}$, let~$L$ be the leaf of~$\mathcal{F}$ passing through~$p$ and suppose that the affine structure of~$L$ at~$p$ is nonsingular. Let~$T$ be a transversal of~$\mathcal{F}$ through~$p$, let~$\Sigma=T\cap \mathfrak{U}$. The restricted holonomy representation~$\mathrm{hol}:\pi_1(L,p)\to \mathrm{Homeo}(\Sigma,p)$ factors through the monodromy representation~$\mathrm{mon}:\pi_1(L,p)\to \mathrm{Aff}(\mathbf{C})$, this is, $\mathrm{hol}$ is a homomorphism of the image of~$\pi_1(L,p)$ under~$\mathrm{mon}$ and, in particular, if an element of~$\pi_1(L,p)$ has trivial monodromy, it has trivial restricted holonomy as well.
\end{lemma}

\subsection{Parabolic ends of leaves}
The leaves of~$\mathcal{F}$ are not closed in general, but some ends of some leaves may have a particularly mild dynamic behavior. Let~$L$ be a leaf of~$\mathcal{F}$. Partially borrowing the terminology from~\cite{brunella-bouts}, we will say that an end of~$L$ (considering the latter as a curve) is  an \emph{analytic parabolic end} if there exists a holomorphic map $\gamma:\Delta\to S$ (for~$\Delta=\{z; |z|<1\}$) such that, for~$\Delta^*=\Delta\setminus\{0\}$, $\gamma(\Delta^*)\subset L$, $\gamma(0)\notin L$, and such that~$\gamma|_{\Delta^*}$ is a biholomorphism between~$\Delta^*$ and a neighborhood of the corresponding end of~$L$. Such an analytic parabolic end of~$L$ (considered as a curve) may be compactified by adding~$\gamma(0)$ and declaring that~$\gamma$ is a local parametrization.

\subsection{Reduced foliations and vector fields}\label{bir} 

The objects that concern our Main Theorem transform naturally under birational transformations of the ambient manifold. We will use these transformations it in order to simplify the local aspects of our study.

A foliation~$\mathcal{F}$ on a surface is said to be \emph{reduced in Seidenberg's sense} if, in the neighborhood every point~$p$, either $\mathcal{F}$ is nonsingular at~$p$, or $\mathcal{F}$ is tangent to a holomorphic vector field vanishing at~$p$  whose linear part  either (i)~has two nonzero eigenvalues whose ratio is not a positive rational,  or (ii)~has one zero eigenvalue but is not nilpotent. Seidenberg's theorem affirms that every foliation becomes reduced after a locally finite number of blowups.

A meromorphic vector field~$X$ on a complex surface is said to be \emph{reduced} if~$\mathcal{F}$, the foliation it induces, is reduced in Seidenberg's sense and if, at every point~$p$, the curve given by the union of the curve of zeros and poles of~$X$ and the union of the separatrices of~$\mathcal{F}$ through~$p$ (if~$p$ is a singularity of~$\mathcal{F}$) or the integral curve of~$\mathcal{F}$ through~$p$ (if~$\mathcal{F}$ is nonsingular at~$p$), has normal crossings. A meromorphic vector field may be transformed, by a locally finite number of blowups, to a reduced one~\cite[section~5]{guillot-rebelo}. 

\subsection{Special foliations} Some special foliations adapted to rational or elliptic fibrations will play a prominent role in our discussion. We refer the reader to~\cite[ch.~4]{brunella-birational} for details.

\subsubsection{Riccati foliations}\label{subsec-ric} A foliation~$\mathcal{F}$ on a compact complex surface~$S$ is a \emph{Riccati foliation} is there exists a rational fibration $\Pi:S\to B$ (possibly with singular fibers) that is \emph{adapted} to~$\mathcal{F}$ in the sense that~$\mathcal{F}$ is transverse to the generic fiber of~$\Pi$. The fibers which are not transverse to~$\mathcal{F}$ are called \emph{special}. If~$p_1,\ldots,p_k$ are points in~$B$ corresponding to the special fibers, there is a \emph{monodromy}\footnote{Warning!!! we are using the word \emph{monodromy} for two different concepts. We have the monodromy representation associated to a developing map of an affine structure and the monodromy representation of a Riccati (or turbulent) foliation.} representation~$\rho:\pi_1(B\setminus\{p_1, \ldots, p_k\})\to \mathrm{PSL}(2,\mathbf{C})$, from which one can recover many features of~$\mathcal{F}$.

Some natural birational transformations for Riccati foliations are given by \emph{elementary transformations} of the underlying ruled surface: if~$F$ is a fiber in the neighborhood of which~$\Pi$ is a locally trivial fibration, one may blow up a point in~$F$, creating an exceptional divisor~$D$ of self-intersection~$-1$ and making~$F$ a curve of self-intersection~$-1$, which may then be blown down, making~$D$ a curve of self-intersection~$0$. In coordinates~$(z,w)$ in~$\Delta\times \mathbf{P}^1$,
\begin{equation}\label{flip}(z,w)\mapsto(z,zw)\end{equation}
is such an elementary transformation.  Up to a birational transformation and in suitable coordinates, a Riccati foliation is, in the neighborhood of a fiber, of one of the following kinds (see~\cite[ch.~4, prop.~2]{brunella-birational}):
\begin{description}
	\item [Nonsingular or transverse] Those given by~$dw$.
	\item [Nondegenerate, nonparabolic (hyperbolic/elliptic)] Given by~$\lambda wdz-zdw$, with $\lambda\in\mathbf{C}\setminus \mathbf{Q}$. The monodromy is~$w\mapsto e^{2i\pi\lambda}w$. We may further distinguish among the \emph{hyperbolic} ($\lambda\in\mathbf{C}\setminus\mathbf{R}$) and \emph{elliptic} ($\lambda\in\mathbf{R}\setminus\mathbf{Q}$) kinds. The elementary transformation~(\ref{flip}) replaces~$\lambda$ by~$\lambda+1$. 
	\item[Nondegenerate, parabolic] Those given by~$dz-z dw$. The monodromy is~$w\mapsto w+2i\pi$.
	\item[Dicritical] Those given by the forms
	\begin{equation}\label{ricdic} p w \,dz-qz\,dw 
	\end{equation}
	for some relatively prime integers~$p$ and~$q$, $q>p>0$. The foliation has the local first integral~$w^q/z^p$.  The monodromy   $w\mapsto e^{2i\pi p/q}w$  is periodic with period~$q$. The integer~$q$ is the \emph{multiplicity} of the fiber. The generic leaf of~(\ref{ricdic}) may be parameterized by~$t\mapsto (t^q,ct^p)$; its projection onto the base is a $q$-fold ramified cover.
	
	These can be seen as quotients of Riccati foliations in the neighborhood of a transverse fiber under the action of a cyclic group. To this effect, consider coordinates~$(x,y)$ in~$\Delta\times\mathbf{P}^1$ and the Riccati foliation~$\mathcal{G}$ given  by~$dy=0$. Let~$\eta$ be a primitive~$q$th root of unity and consider the action of~$\mathbf{Z}/q\mathbf{Z}$ generated by~$(x,y)\mapsto(\eta x,\eta^p y)$, which preserves~$\mathcal{G}$. The map~$(x,y)\mapsto(x^q,x^p/y)$ is invariant under this action and maps~$\mathcal{G}$ to the dicritical foliation~(\ref{ricdic}). We may choose to reduce the singularities of the foliation at the price of complicating the space. Upon reducing the singularity of~(\ref{ricdic}) at~$(0,0)$ we will find a chain of rational curves, all of them invariant by the foliation except one, transverse to the foliation, which is not one of the two extremal ones. Each one of the two invariant chains can be contracted to produce a cyclic quotient singularity. The resulting singular space is the quotient of~$\Delta\times\mathbf{P}^1$ under the above action of~$\mathbf{Z}/q\mathbf{Z}$.
	
	\item[Semidegenerate] There are two saddle-node singularities on~$F$ whose strong separatrices are contained in~$F$  (in particular, the holonomy of the foliation along~$F$ is tangent to the identity). The weak separatrices may or may not converge. 
	
	\item[Nilpotent] There is only one singular point in~$F$, with nilpotent linear part. These can be seen as quotients of semidegenerate models by an involution exchanging the two saddle-nodes. Its resolution has three rational divisors (including~$F$): one central divisor of self-intersection~$-1$ having one saddle-node singularity (with strong separatrix contained in the divisor) and two saddles where the other two divisors, of self-intersection~$-2$, intersect (there are no further singularities of~$\mathcal{F}$ in these). 
\end{description}
A Riccati foliation having special fibers as above will be said to be in~\emph{standard form}.

\paragraph{Suspensions} Let~$E$ be a curve with universal cover~$\widetilde{E}$, $\rho:\pi_1(E)\to\mathrm{Aut}(\mathbf{P}^1)$ a representation.  There is a natural diagonal action of~$\pi_1(E)$ on~$\widetilde{E}\times \mathbf{P}^1$, by deck transformations in the first factor, via~$\rho$ in the second one. Let~$S$ be the quotient. It is the \emph{suspension} of the representation~$\rho$. From the two natural foliations in~$\widetilde{E}\times \mathbf{P}^1$  we have, in~$S$, a rational fibration~$\Pi:S\to E$ and a foliation on~$S$ transverse to it (this foliation is  a Riccati one). 

If~$X$ is a vector field inducing a Riccati foliation~$\mathcal{F}$ on the surface~$S$, the divisor of zeros and poles of~$X$ is of the form~$\sum_{j}{a_j}F_j+\sum_i b_i E_i$, where~$F_j$ is a fiber (either special or transverse) and where the divisor $E_i$ is irreducible and not supported on a fiber (it may or may not be invariant by~$\mathcal{F}$). For a generic fiber~$H$, with respect to the tangent bundle~$T_\mathcal{F}$ of~$\mathcal{F}$, $T_\mathcal{F}\cdot H=0$ (see~\cite[ch.~4, section~1]{brunella-birational}), and since~$T_\mathcal{F}$ is given by the previous divisor, 
\begin{equation}\label{trzer} \left(\sum_i b_i E_i\right)\cdot H =\left(\sum_{j}{a_j}F_j+\sum_i b_i E_i\right)\cdot H=T_\mathcal{F}\cdot H =0.\end{equation}

\subsubsection{Elliptic fibrations, turbulent foliations} An \emph{elliptic fibration} on a surface~$S$ is a map~$\Pi:S\to B$ onto a curve, with connected fibers, such that the generic fiber is an elliptic curve. Kodaira classified, as divisors, the fibers of elliptic fibrations that are not elliptic curves, up to birational transformations. Kodaira's combinatorial models bear the symbols~$\mathrm{I}_n$ ($n\geq 0$) $\mathrm{II}$, $\mathrm{III}$, $\mathrm{IV}$, $\mathrm{I}_n^*$ ($n\geq 0$) $\mathrm{II}^*$, $\mathrm{III}^*$, $\mathrm{IV}^*$.  They are all \emph{divisors of elliptic fiber type}, effective divisors of the form~$D=\sum a_i D_i$ such that $D\cdot D_i=0$ and~$K_S\cdot D=0$ (see~\cite{bpv} for details). 

A foliation~$\mathcal{F}$ on a compact complex surface~$S$ is a \emph{turbulent foliation} if there exists an elliptic fibration $\Pi:S\to B$ that is \emph{adapted} to~$\mathcal{F}$ in the sense that~$\mathcal{F}$ is transverse to the generic fiber of~$\Pi$.  
A turbulent foliation comes with a \emph{monodromy} representation~$\rho:\pi_1(B\setminus\{p_1, \ldots, p_k\})\to \mathrm{Aut}(F)$.

In the neighborhood of a fiber, up to a ramified cover of the base and a bimeromorphic transformation, a turbulent foliation is nonsingular and adapted to a locally trivial elliptic fibration. In a turbulent foliation of a locally trivial elliptic fibration, a fiber may be either everywhere transverse to the foliation or totally invariant by it. This implies that, for a general turbulent foliation, we have two coarse kinds of fibers, \emph{transverse} (or \emph{dicritical}) and \emph{invariant}.

\subsection{The Brunella-Mendes-McQuillan classification of foliations}\label{sec:bmm}  In a series of works, Brunella, Mendes and McQuillan established a classification of foliations on algebraic surfaces in the spirit of Kodaira's classification of algebraic surfaces. Given a foliation~$\mathcal{F}$ on an algebraic surface~$S$, the \emph{Kodaira dimension of~$\mathcal{F}$} is defined, $\mathrm{kod}(\mathcal{F})\in\{-\infty,0,1,2\}$. The classification concerns those foliations whose Kodaira dimension is less than two. We refer the reader to~\cite{brunella-birational} and to~\cite{brunella-pisa} for a detailed presentation. For our study, we will use two remarkable outcomes of this classification: the description of foliations on algebraic surfaces admitting an invariant elliptic curve (theorem~\ref{brunella-elcurve}) and of those admitting a tangent entire curve (theorem~\ref{mq}).

\section{A dichotomy}
In the algebraic case, the proof of the Main Theorem will be separately made for each one of the two cases arising from the following dichotomy:
\begin{proposition}\label{dichotomy} Let~$S$ be a compact complex surface, $X$ a   meromorphic vector field on~$S$, $P\subset S$ the curve of poles of~$X$, $\Omega\subset\mathbf{C}$ an open set and~$\phi:\Omega\to S\setminus P$ a univalent maximal solution of the restriction of~$X$ to~$S\setminus P$. Either:
	\begin{itemize}
		\item $\Omega=\mathbf{C}$, or~$\phi$ extends holomorphically to a map~$\widehat{\phi}:\mathbf{C}\to S$; or
		\item there exists~$p\in M$,  $p\in\overline{\phi(\Omega)}\cap P$ such that $\overline{\phi(\Omega)}$ is not contained in any proper analytic subset in a neighborhood of~$p$. 
	\end{itemize}
\end{proposition}

\begin{proof} Let~$L=\phi(\Omega)$. If~$L$ does not accumulate to~$P$, there exists a neighborhood~$U$ of~$P$ that does not intersect~$L$. By the existence theorem for the solutions of ordinary differential equations and the compactness of~$S\setminus U$, there exists some~$\epsilon>0$ such that any solution of~$X$ with initial condition in~$S\setminus U$ is defined for~$|t|<\epsilon$. In particular, since~$\phi$ is univalent maximal and~$L\cap U=\emptyset$, $\Omega=\mathbf{C}$. Hence, if~$\Omega\neq\mathbf{C}$, there exists at least one connected component~$D_0$ of~$P$ such that~$L$ intersects any neighborhood of~$D_0$. Let~$\{t_i\}\subset\Omega$ be a sequence converging to~$t_\infty\in\mathbf{C}\setminus\Omega$ such that~$\{\phi(t_i)\}$ converges to~$p\in  P$.

	Let~$V$ be an analytic subset in a neighborhood~$U$ of~$p$, $p\in V$, containing~$\overline{L}\cap U$, and let~$V_0\subset V$ be an irreducible component containing infinitely many~$\phi(t_i)$. By Puiseux's theorem on the parametrization of analytic sets, there exists a holomorphic parametrization~$\gamma:(\Delta,0)\to (V_0,p)$. It compactifies one analytic parabolic end of~$L$. Through~$\gamma$, $X$ induces a meromorphic vector field~$X_0$ on~$\Delta^*$ (which, by lemma~\ref{fewstat}, extends holomorphically to~$0$). If~$X_0(0)=0$, in~$\Delta$, the solution of~$X_0$ takes infinite time to reach~$0$ (as explained in the proof of lemma~\ref{fewstat}). This contradicts the fact that~$\{t_i\}$ converges in~$\mathbf{C}$. We must conclude that~$X_0(0)\neq 0$. A local parametrization of the solution of~$X_0$ through~$p$ allows~$\phi$ to be extended to a neighborhood of~$t_\infty$ (which, a posteriori, turns out to be an isolated point of the complement of~$\Omega$).
	
	In this way, if~$\phi$ cannot be extended to a neighborhood of~$t_\infty$, $\overline{\phi(\Omega)}$ is not contained in any proper analytic subset of~$S$ in a neighborhood of~$p$.\end{proof}

In what follows, for~$S$   a  surface (not necessarily compact), $X$ a reduced meromorphic vector field on~$S$ and~$\phi:\Omega\to S$ a univalent maximal solution, we will let $L=\phi(\Omega)$ and~$\widehat{\mathcal{L}}$ be the curve obtained from~$L$ after adding all the analytic parabolic ends. By lemma~\ref{fewstat}, the restriction of $X$ to~$L$ extends holomorphically to~$\widehat{\mathcal{L}}$ and has at most two zeros (counted with multiplicity). We will denote by~$\mathcal{L}\subset\widehat{\mathcal{L}}$ the curve where the induced vector field is nonvanishing, $L\subset\mathcal{L}\subset\widehat{\mathcal{L}}$. Upon adding some isolated points to~$\Omega$ we can extend~$\phi$ to obtain a parametrization~$\phi:\Omega\to \mathcal{L}$. (There is no contradiction with the maximal character of~$\phi$: the maximal solution~$\phi$ may not be maximal as a function from a subset of~$\mathbf{C}$ into~$S$.) 

In section~\ref{sec:combi} we will prove the Main Theorem for the first possibility where proposition~\ref{dichotomy} leads; we will do it for  the second one in  section~\ref{sec:entire}.

\section{Local models at the limit set}

We will now give local models of reduced meromorphic vector fields in the neighborhood of points where~$\mathfrak{U}$ accumulates in a nonanalytic way (these appear, for instance, with the second possibility of proposition~\ref{dichotomy}). The result that follows is an analogue of proposition~17 in~\cite{guillot-rebelo}, which concerns semicomplete vector fields. The local models that we will obtain here are essentially the same as the ones obtained there, although our hypothesis are different. The proof  will follow a path not too far from and at times intersecting that of~\cite[prop.~17]{guillot-rebelo}.

\begin{proposition}\label{pr:norfor} Let~$S$ be a (not necessarily compact) surface, $X$ a reduced meromorphic vector field on~$S$, $P$ its curve of poles and~$\phi:\Omega\to S\setminus P$ a univalent maximal solution. Let~$L=\phi(\Omega)$ and let~$p\in \overline{L}$. If there does not exist a neighborhood~$U$ of~$p$ such that~$\overline{L}\cap U$ is analytic, either
	\begin{itemize}\item $X$ is holomorphic at~$p$, where it does not vanish or where it has an isolated zero; or
		\item in suitable coordinates and up to multiplication by a nonvanishing holomorphic function, $X$~is locally of one the forms of table~\ref{tab:norfor}.
	\end{itemize}
\end{proposition}

\begin{table}
	\begin{tabular}{|c|c|c|c|c|}\hline
		name & local model & ord & ind & CS \\ \hline \hline
		
		regular & $\displaystyle y^q\del{x}$ & $q$ & $1$ & $0$ \\ \hline
		
		\multirow{4}{*} {\begin{tabular}{c}finite \\ ramification \end{tabular}} & \multirow{4}{*} {\begin{tabular}{c} $\displaystyle x^py^q\left(mx\del{x}-ny\del{y}\right)$ \\ $pm-qn=\pm1$ \end{tabular}} & \multicolumn{3}{|c|}{$x=0$}   \\ \cline{3-5}
		&    & $p$ & $n$ & $-m/n$   \\ \cline{3-5}
		&    &  \multicolumn{3}{|c|}{$y=0$} \\ \cline{3-5}
		&    & $q$ & $-m$ & $-n/m$   \\ \cline{1-5}
		
		\multirow{4}{*} {\begin{tabular}{c}infinite \\ ramification \end{tabular}} & \multirow{4}{*} {\begin{tabular}{c} $\displaystyle
				(x^py^q)^r\left(x[q+\cdots]\del{x}-y[p+\cdots]\del{y}\right)$ \\ $p,q>0$, $(p,q)=1$, $r\neq 0$ \end{tabular} } & \multicolumn{3}{|c|}{$x=0$}   \\ \cline{3-5}
		&     & $rp $ & $\infty$ & $-q/p$   \\ \cline{3-5}
		&    &  \multicolumn{3}{|c|}{$y=0$} \\ \cline{3-5}
		&    & $rq $ & $\infty$ & $-p/q$   \\ \hline
		saddle node & $\displaystyle y^q\left([x+\cdots]\del{x} +y^{k+1}\del{y}\right)$ & $q$ & $\infty$ & $0$    \\ \hline
	\end{tabular}
	\caption{The (semi) local models for proposition~\ref{pr:norfor} (up to multiplication by a nonvanishing holomorphic function). In these, $k,m,n, p,q,r\in\mathbf{Z}$, $m,n>0$ and~$k\geq 0$. } \label{tab:norfor}
\end{table}

We include in table~\ref{tab:norfor} further information concerning the invariant curves of the local models: the orders of the vector fields along them, the ramification indices of the induced affine structures and the local contributions to the self-intersection of the curves according to the Camacho-Sad theorem~\cite{camacho-sad}.

Let us start  the proof of proposition~\ref{pr:norfor}. Let~$S$ be a (not necessarily compact) surface, $X$ a reduced meromorphic vector field on~$S$, $\mathcal{F}$ the foliation it generates and~$\phi:\Omega\to S$ a univalent maximal solution of~$X$, $L=\phi(\Omega)$ and~$\widehat{\mathcal{L}}$ the curve obtained from~$L$ after adding all the analytic parabolic ends. Let~$w\in S$ be such that there exists a sequence~$\{w_i\}\subset \widehat{\mathcal{L}}$ such that~$\lim_{i\to\infty}w_i=w$ but such that no~$w_i$ belongs to a separatrix of~$\mathcal{F}$ through~$w$. Consider a coordinate system centered at~$w$. The discussion splits naturally into the following cases, according to the local nature of~$\mathcal{F}$:

\subsection{Regular point} Suppose that~$\mathcal{F}$ is regular at~$0$. Under the hypothesis that the vector field is reduced, coordinates can be chosen in a way such that $X=f(x,y)x^py^q\indel{x}$, with~$p,q\in\mathbf{Z}$ and~$f$ a holomorphic function that does not vanish at~$0$. The leaf~$\widehat{\mathcal{L}}$ intersects infinitely many times~$\{x=0\}$, and contains a sequence of points of the form~$(y_i,0)$, $\lim_{i\to \infty}y_i=0$. Within~$\widehat{\mathcal{L}}$ and in restriction to a neighborhood of~$(y_i,0)$, the vector field is~$f(x,y_i)x^p\indel{x}$. If~$p\neq 0$, $\widehat{\mathcal{L}}$ has infinitely many zeros or poles at these points, and, by lemma~\ref{fewstat}, cannot belong to~$\mathfrak{U}$. We must conclude that~$p=0$.

\subsection{Nondegenerate singularity} If~$\mathcal{F}$ has a nondegenerate singularity at~$0$ (if it can be generated by a vector field with an isolated singularity and two nonzero eigenvalues), in suitable coordinates, the vector field is of the form
\begin{equation}\label{vf}X=x^py^q\left(\lambda x f(x,y)\del{x}+ \mu y g(x,y)\del{y}\right),\end{equation}
with $\mu/\lambda \in\mathbf{C}^*\setminus \mathbf{Q}^+$, $p,q\in\mathbf{Z}$, $p\neq 0$ or~$q\neq 0$ and $f$ and $g$ holomorphic functions that do not vanish at the origin and such that~$f(0)=g(0)$. The separatrices are given by~$\{ x=0\}$ and~$\{y=0\}$.

\subsubsection{In the Siegel domain} This is the case where~$\mu/\lambda\in\mathbf{R}^-$. It is not difficult to see that, locally, \emph{every closed set invariant by the foliation, containing the origin, that is not contained in the union of the separatrices, contains both separatrices} (see~\cite[lemma~18]{camacho-rosas}). Thus, $\mathcal{L}$ contains both separatrices in its closure, which belongs to~$\mathfrak{U}$. Let us study the affine structure induced by~(\ref{vf}) on the separatrix~$\{y=0\}$. Let us suppose that~$\lambda=1$. We have, from the proof of proposition 17 in~\cite{guillot-rebelo}, that
\begin{equation}\label{ramind}\mathrm{ind}(\{y=0\},0)=-\frac{1}{p+\mu q},\;\mathrm{ind}(\{x=0\},0)=-\frac{\mu}{p+\mu q}.\end{equation}

If $p+\mu q=0$ then both ramification indices equal~$\infty$, and $\mu=-p/q\in\mathbf{Q}$ (in particular, $p$ and~$q$ have the same sign). This gives the \emph{infinite ramification} local model.

If~$p+\mu q\neq 0$, $\mathrm{ind}(\{y=0\},0)\neq \infty$ and  thus $\mathrm{ind}(\{y=0\},0)\in\mathbf{Z}^*$. In particular, $\mu\in\mathbf{Q}^-$. Let~$\mu=n/m$
with~$(m,n)=1$ (notice that~$m$ and~$n$ must have different signs) and let~$\Delta=pm+qn\in\mathbf{Z}^*$ so that~$\mathrm{ind}(\{y=0\},0)=-m/\Delta$ and~$\mathrm{ind}(\{x=0\},0)=-n/\Delta$. Since~$\Delta$ divides $m$ and~$n$, $\Delta^2=1$. If~$q<0$ and~$p\geq 0$ then if~$pm+qn=1$, we necessarily have~$n=-1$ and~$p=0$: the curve of zeros and the curve of poles do not intersect at such a point. The monodromy of the affine structure along the separatrix~$\{y=0\}$ is periodic with period~$m$.  If~$T$ is a curve transverse to the separatrix then, by  lemma~\ref{holmon},  the holonomy of~$\mathcal{F}$ along the separatrix has periodic points of periods dividing~$m$ at the points of~$\mathfrak{U}\cap T$. The existence of periodic points of these periods forces the holonomy to be periodic with period~$m$. On its turn, by the results of Mattei and Moussu~\cite[Th\'eor\`eme~2]{mattei-moussu} (or by the theorem of Martinet and Ramis~\cite{MR-res}), this implies that the foliation is linearizable, that the vector field is of the form
\begin{equation}\label{vf-lin}X=f(x,y)x^py^q\left(mx\del{x}+ n y \del{y}\right),\end{equation}
the \emph{finite ramification} local model\footnote{We can further simplify the expression of the vector field and suppose, in~(\ref{vf-lin}) that~$f\equiv 1$ (although we will not use this refinement in what follows). Let us sketch a proof of this fact following the proof of proposition~18 in~\cite{guillot-rebelo}. Consider the \emph{period function}, the function that to each orbit of~(\ref{vf-lin}) other that the separatrices associates the period of the solutions parametrizing it. The proof of~\cite[prop.~18]{guillot-rebelo} shows that this function extends as a meromorphic function to the orbit space of the vector field and that we can suppose that~$f\equiv 1$ in~(\ref{vf-lin}) if this period function is identically zero. It thus suffices to prove that, in restriction to~$\mathfrak{U}$, the period function is identically zero. Consider the path~$\gamma:[0,1]\to \mathbf{C}^2$, $\gamma(t)=(e^{2i\pi m},0)$. It winds~$m$ times around the separatrix~$y=0$, where the affine structure has  ramification index~$m$. Under the developing map of the affine structure, the image of~$\gamma$ maps to a closed curve. We can lift~$\gamma$ to a curve~$\widetilde{\gamma}:[0,1]\to \mathbf{C}^2$ taking values in a leaf~$\widetilde{L}$ within~$\mathfrak{U}$ which is close enough to~$\{y=0\}$. We can do so in a way that is locally the restriction of an isomorphism of the corresponding affine structures, like in the proof of lemma~\ref{holmon}. This implies that the developing map of the affine structure in~$\widetilde{L}$ maps the image of~$\widetilde{\gamma}$ to a closed curve. On its turn, this implies that the integral of the time form along the image of~$\widetilde{\gamma}$ in~$\widetilde{L}$ vanishes. Since~$\widetilde{L}$ belongs to~$\mathfrak{U}$, $\widetilde{\gamma}$ is a closed curve: the period function vanishes identically on~$\widetilde{L}$.}.

\subsubsection{In the Poincar\'e domain}\label{poincare} This is the case where~$\mu/\lambda \notin\mathbf{R}^-$, $\mu/\lambda \notin\mathbf{Q}^+$. By Poincar\'e's linearization theorem, the foliation is linearizable, so we may suppose that~(\ref{vf}) is actually of the form 
\begin{equation}\label{vfpoin}X=f(x,y) x^py^q\left(\lambda x\del{x}+\mu y\del{y}\right),\end{equation}
with~$f$ a holomorphic function such that~$f(0,0)=1$. If both~$p$ and~$q$ vanish, the vector field is holomorphic and has an isolated singularity at~$0$ (one of the possibilities of proposition~\ref{pr:norfor}) so we will, from now on, suppose that~$p$ and~$q$ do not vanish simultaneously. We will prove that such vector fields do not appear in the context of the proposition.

We will first address the case where~$f\equiv 1$. Up to multiplying~$X$ by a constant, we will suppose that~$\Re(\lambda)> 0$ and~$\Re(\mu)> 0$ (this is possible because we are in the Poincar\'e domain). From~(\ref{ramind}), in each one of the separatrices the affine structure has an irrational ramification index and thus the affine structure is  not uniformizable. (However, this lack of uniformizability does not give a direct obstruction for the uniformizability of the affine structures of other orbits, for some orbits may not accumulate to the separatrices: for example, if~$\lambda$ and~$\mu$ are real we have the real-valued first integral~$|x|^\mu|y|^{-\lambda}$.) Let~$V_N=\{z| \Re(\lambda z)<-N, \Re(\mu z)<-N\}$ and consider the mapping~$\rho:V_N\to\mathbf{C}^2$, $\rho(\zeta)=(x_0 e^{\lambda\zeta}, y_0e^{\mu\zeta})$, parametrizing injectively some subset of an orbit~$L$ of~$X$. For any neighborhood~$U$ of the origin of~$\mathbf{C}^2$ there exists~$N>0$ such that~$\rho(V_N)$ is in~$L\cap U$. In restriction to~$\rho(V_N)$, the vector field is, up to a constant factor,
$e^{(\lambda p+\mu q)\zeta}\indel{\zeta}$ (notice that~$\lambda p+\mu q\neq 0$). Integrating the time form into the \emph{time  function}~$T$, we find that 
\[T(\zeta)= \int e^{-(\lambda p+\mu q)\zeta} d\zeta  =-\frac{e^{-(\lambda p+\mu q)\zeta}}{\lambda p+\mu q}.\]
There is no value of~$N$ such that the restriction of~$T$ to $V_N$ is injective. We conclude that \emph{in~$\mathbf{C}^2\setminus\{0\}$, there is no orbit of~$X$ carrying a uniformizable affine structure.}

Let us now proceed to the general case. Let~$B=\{(x,y); |x|^2+|y|^2<1\}\subset\mathbf{C}^2$, $B^*=B\setminus\{0\}$ and suppose that~$f$ is defined in~$B$. Consider, for~$\alpha\in(0,1]$, the transformations~$h^\alpha(x,y)=(x/\alpha,y/\alpha)$. Notice that~$h^\alpha$ preserves the foliation induced by~$X$. Consider the vector field
\[X_\alpha =(\alpha^{p+q})^{-1}h^\alpha_* X =f(\alpha  x,\alpha y)x^py^q\left(\lambda x\del{x}+\mu y\del{y}\right).\]
Let~$X_0=\lim_{\alpha\to 0}X_\alpha$ (the limit exists because~$\Re(\lambda)>0$ and~$\Re(\mu)>0$). When~$\alpha\neq 0$, the vector field~$X_\alpha$ is conjugated to~$X$, up to a constant multiple. When~$\alpha=0$, it corresponds to~$f\equiv 1$ in~(\ref{vfpoin}). The vector field varies continuously with~$\alpha$. We may define, in the real manifold-with-boundary~$[0,1]\times B^*$, the real two-dimensional foliation~$\mathcal{G}$ given in~$\{\alpha\}\times B^*$ by the restriction of the foliation induced by~$X_\alpha$. Each leaf of~$\mathcal{G}$ carries a natural affine structure, which varies continuously with~$\alpha$. An argument analogous to the one used to prove proposition~\ref{unifisclosed} shows that the set of points belonging to leaves of~$\mathcal{G}$ which do not have a uniformizable affine structure is open.\footnote{We could also consider the real and imaginary parts of~$X$ as real commuting vector fields on a real space of dimension four and use, for instance, the methods of the proof of Theorem~A in~\cite{arroyo-guillot}.} We have already proved that there is no leaf of~$X_0$ having a  uniformizable affine structure. There is thus an open subset~$W$ of~$\{0\}\times B^*$ within~$[0,1]\times B$ saturated by~$\mathcal{G}$ and consisting of  leaves that do not carry a uniformizable affine structure.  The set~$W$ contains, for~$S^3=\{(x,y); |x|^2+|y|^2=1/2\}$ and some~$\epsilon>0$, the set~$[0,\epsilon]\times S^3$.  Since~$X_\alpha$ belongs to the Poincar\'e domain, the saturation of~$S^3$ by~$X_\alpha$ is~$B^*$. This shows that for~$\delta<\epsilon$ there is no leaf of~$X_\delta$ having a uniformizable affine structure.  We thus prove that vector fields in the Poincar\'e domain satisfy the hypothesis of proposition~\ref{pr:norfor} only when they are holomorphic and have an isolated singularity.

\subsection{Saddle-node}\label{sec:saddle-node} If~$w$ is a saddle-node singularity of~$\mathcal{F}$, in suitable coordinates
\[X=h(x,y)(x-g(y))^ny^q\left(f(x,y)\del{x}+y^{k+1}\del{y}\right),\]
with~$k\geq 1$, $f$ a holomorphic function such that~$f(x,y)=x+\cdots$,  $g$ a holomorphic function such that $g(0)=0$, $h$ a nonvanishing holomorphic function and $q,n\in\mathbf{Z}$. The term~$(x-g)^n$ accounts for a curve of zeros or poles transverse to the strong separatrix~$\{y=0\}$. We claim that~$n=0$ (this would give the sought local model). As remarked in~\cite[lemma~19]{camacho-rosas}, the description of the saddle-node singularity by sectorial normalizations in~\cite{HKM} implies that \emph{with the exception of the central manifold (when it converges), every integral curve accumulates to the strong separatrix of the saddle-node}. Hence, $\mathcal{L}$ accumulates to the strong separatrix~$\{y=0\}$, which belongs to~$\mathfrak{U}$. For the affine structure along the latter, from the proof of proposition~17 in~\cite{guillot-rebelo}, $\mathrm{ind}(\{y=0\},0)=-1/n$, and thus~$n\in\{-1,0,1\}$. If~$n=\pm 1$ then, by lemma~\ref{holmon}, the holonomy has fixed points at the points of~$\mathfrak{U}$. However, the holonomy of the strong separatrix is tangent to the identity and has~\cite[section~5.3.3]{loray} a Leau-Fatou flower dynamics: it cannot have such a set of fixed points. This proves that~$n=0$, giving the sought normal form. In particular, \emph{the central manifold, if convergent, does not belong to the curve of zeros and poles}.

This finishes the proof of proposition~\ref{pr:norfor}.

\section{In the presence of nonentire solutions}\label{sec:combi}

We will now prove the Main Theorem in the case where $\phi:\Omega\to S$ is a univalent maximal solution of a meromorphic vector field on the algebraic compact complex surface~$S$ that does not extend to~$\mathbf{C}$ (one of the two cases where proposition~\ref{dichotomy} leads). Mostly, the proof follows that of Theorem~B in~\cite{guillot-rebelo}, and many of its lines will only be sketched. The aim is, first, to describe a connected component of the divisor of poles where the image of the solution accumulates (theorem~\ref{thamabis} below), and then to use this description along with global information on the surface to produce a fibration adapted to the vector field.

\subsection*{The divisor of poles} We have the following result, direct analogue of Theorem~A in~\cite{guillot-rebelo}:

\begin{theorem}\label{thamabis} Let~$S$ be a complex surface, $X$ a meromorphic vector field on~$S$, $D\subset S$ a compact connected component of the curve of poles of~$X$, $\Omega\subset\mathbf{C}$, and~$\phi:\Omega\to S$ a univalent maximal solution of~$X$ that accumulates to~$D$ in a nonanalytic way. Then~$D$ is invariant by the foliation induced by~$\mathcal{F}$ and, up to a birational transformation, either
	
	\begin{itemize}
		\item $D$ can be collapsed to a point where~$X$ becomes holomorphic; 
		\item $D$ is a rational curve of vanishing self-intersection;   
		\item $D$ is a nonsingular elliptic curve of vanishing self-intersection or, more generally, supports a divisor  of elliptic fiber type. 
	\end{itemize}

\end{theorem}
The theorem concerns only a neighborhood of~$D$ and, for example, does not require the compactness of~$S$. For the proof, after reducing the vector field (as defined in section~\ref{bir}), the vector field is locally described by the local models of proposition~\ref{pr:norfor}. Each   irreducible component of~$D$  is invariant under~$\mathcal{F}$. It has a self-intersection number; the vector field has some vanishing order along it, and the affine structure it carries is one of those listed in proposition~\ref{affinecomclas}. All of these objects come with some numerical invariants that are not independent. The above theorem follows from investigating the combinatorial relations between all these data while taking into account the fact that it may be possible to blow down a curve without affecting the reduced nature of the vector field. The associated combinatorics are the same as those in the proof of Theorem~A in~\cite[section~6]{guillot-rebelo}, which can be used word by word to give a proof the above result.

The combinatorial analysis leading to theorem~\ref{thamabis} gives also a detailed description of the foliation in the neighborhood of the divisor in the second and third cases.  For instance, if~$D$ is a nonsingular elliptic curve or supports a divisor of elliptic fiber type, $\mathcal{F}$ does not have singularities at the regular points of~$D$ and, for a singular point (normal crossing)~$p$ of~$D$, the separatrices of~$\mathcal{F}$ through~$p$ are contained in~$D$; if~$D$ is a rational curve of vanishing self-intersection,  either~$\mathcal{F}$ has two singularities at~$D$, each one being a saddle-node whose strong separatrix is contained in~$D$, or~$\mathcal{F}$ has only one singularity at~$D$, having a nilpotent linear part, whose resolution has the combinatorics of the nilpotent special fiber of a Riccati foliation (section~\ref{subsec-ric}).

\subsection*{Constructing the fibrations}

Having at hand divisors that are natural candidates to be fibers of the fibration announced in the Main Theorem, one must now build the fibration out of these divisors: one must show that there exists a (rational or elliptic) fibration on~$S$ having~$D$ (or, more precisely, the divisor of vanishing self-intersection~$Z$ supported on~$D$) as a fiber.

For the rational curves of vanishing self-intersection in theorem~\ref{thamabis}, $D$ has a neighborhood biholomorphic to a product~$\Delta\times \mathbf{P}^1$ (this is Savel'ev's theorem~\cite{savelev}), and we thus have a local rational fibration. This local fibration may me extended to a global one (for instance, through theorem~2 in~\cite{jvp-divisors}). Since  the poles of~$X$ are contained in a fiber in a neighborhood of~$D$, the vector field preserves the fibration (by the maximum principle, the flow of~$X$ maps fibers to fibers). The foliation induced by~$X$ is a Riccati one   with respect to the rational fibration. (In this case the algebraicity hypothesis on~$S$ is superfluous and we only need it to be a compact surface.)

The situation is much more subtle for the divisors of elliptic fiber type since, for example, there exist algebraic surfaces having smooth elliptic curves of vanishing self-intersection which are not the fibers of a fibration (see~\cite{sad}). Although, as discussed in~\cite[section~7]{guillot-rebelo}, our problem essentially one of a semilocal nature (it suffices to prove that some neighborhood of~$D$ has nonconstant holomorphic functions), we will mostly resort to global arguments to establish the existence of a fibration. In the case where~$S$ is algebraic, we will prove the Main Theorem in the setting of the third possibility of theorem~\ref{thamabis} by invoking some of Brunella's and McQuillan's results. As an outcome of the classification of holomorphic foliations on algebraic surfaces, which we briefly presented in section~\ref{sec:bmm}, Brunella proved the following~\cite[ch.~9, cor.~2]{brunella-birational}.

\begin{theorem}[Brunella]\label{brunella-elcurve} Let~$S$ be a compact complex algebraic surface, $\mathcal{F}$ a holomorphic foliation on~$S$, $Z$ a divisor of elliptic type on~$S$ whose support is invariant by~$\mathcal{F}$. Then either:
	\begin{enumerate}
		\item   $Z$ is a fiber of an elliptic fibration and either~$\mathcal{F}$ coincides with it (and has a first integral) or is turbulent with respect to it; or
		\item \label{casebrukod0}   there exists an elliptic curve~$E$ and a representation~$\rho:\pi_1(E)\to\mathrm{Aut}(\mathbf{P}^1)$ such that for the associated suspension~$(M,\mathcal{G})$ there is a birational map~$\eta:M'\dashrightarrow M$ and a ramified cover~$r:M'\to S$  such that~$\eta_*r^*\mathcal{F}=\mathcal{G}$.
	\end{enumerate}
\end{theorem}

In the first case, the Main Theorem follows; in the second, it is a consequence of the following result:

\begin{proposition}\label{prop:susp} Let~$X$ be a meromorphic vector field on the surface~$S$ inducing the foliation~$\mathcal{F}$. Suppose that~$X$ has a univalent maximal solution with Zariski-dense  image, and that, up to a birational map and a ramified cover, $\mathcal{F}$ is obtained by suspension over an elliptic curve, this is, that we are in the situation of item~(\ref{casebrukod0}) in theorem~\ref{brunella-elcurve}.  Then, up to a birational transformation of~$S$, either~$X$ is holomorphic or there is  a fibration on~$S$ preserved by~$X$. 
\end{proposition}

In it, we do not assume that the univalent maximal solution of~$X$ approaches the locus of poles of~$X$ in any given way or make hypothesis about the nature of the domain where the solution is defined. The rest of the section will be dedicated to the proof of this proposition.

Let~$E=\mathbf{C}/\Lambda$ and consider a representation~$\rho:\Lambda\to\mathrm{Aut}(\mathbf{P}^1)$. Let~$M$ be the suspension, $\mathcal{G}$ the Riccati foliation on~$M$. The representation~$\rho$ has infinite image, for otherwise~$\mathcal{G}$, and thus~$\mathcal{F}$, would have a first integral. Let~$V$ the meromorphic vector field on~$M$ coming from~$X$ and~$\widetilde{V}$ be the one  on~$\mathbf{C}\times\mathbf{P}^1$ that induces the vector field~$V$ on~$M$. Let~$(t,[z:1])$ be coordinates in~$\mathbf{C}\times\mathbf{P}^1$. In these,  
\begin{equation}\widetilde{V}=\label{campoquito}H(z,t)\del{t}=c(t)\frac{z^n+a_{n-1}(t)z^{n-1}+\cdots+a_0(t)}{z^m+b_{m-1}(t)z^{m-1}+\cdots+b_0(z)}\del{t},\end{equation}
with~$a_i$, $b_i$ and~$c$ meromorphic functions.  

\begin{remark}\label{otherpoles} The poles of~$b_i$ (resp. of~$a_i$) correspond to values of~$t$ where some of the poles (resp. zeros) of~$\widetilde{V}$ escape to infinity, components of the curve of poles  (resp. zeros) of~$\widetilde{V}$ which  intersect the curve~$\{z=\infty\}$ but which are not fibers. 
\end{remark}

The ramified cover~$r$  is a Galois cyclic one (see~\cite[fact~IV.3.3]{mcquillan-canonical} and the comments in~\cite[section~2.2]{jvp-height}), and thus associated to a finite-order birational  automorphism of~$(M,\mathcal{G})$. Any such automorphism~$\sigma$ preserves, by the maximum principle, the fibration (maps fibers to fibers) and acts upon the base~$E$. If it does not have a fixed point on~$E$,  the quotient of~$M$ under~$\sigma$ would still be a suspension over an elliptic curve, so we may restrict to the case where~$\sigma$ has fixed points on~$E$. The cyclic automorphism of~$M$  comes from one of~$\mathbf{C}\times\mathbf{P}^1$ that normalizes the group of transformations of the suspension and preserves~$\widetilde{V}$. If~$\widetilde{V}$ is holomorphic, so will be, up a birational transformation, $X$, and we may thus suppose that~$\widetilde{V}$ is strictly meromorphic.

Let~$\phi:\Omega\to S$ be a maximal solution of~$X$ and let~$L$ be its image. Let~$\widetilde{L}$ be an integral curve of~$\mathcal{G}$ (with its analytic parabolic ends compactified) that maps to~$L$ through~$r$ and~$\eta$ (it is a finite ramified cover of~$L$). 
The solution of~$V$ parametrizing~$\widetilde{L}$ may no longer be single-valued but has, nevertheless, only finitely many determinations: by lemma~\ref{fewstat}, in restriction to~$\widetilde{L}$, $\widetilde{V}$ may only have finitely many zeros.  In particular, the zeros of~$H$ cannot contain vertical curves~$\{t=t_0\}$  for, otherwise, on every transcendent leaf (on every integral curve associated to an infinite orbit of the monodromy), $X$  would have infinitely many zeros. The  vertical components of the curve of zeros correspond to the zeros of~$c$ in~(\ref{campoquito}), and we must conclude that~$c$ has no zeros.   In principle, $\widetilde{L}$ may intersect infinitely many times an irreducible component of the curve of poles of~$V$ (remark~\ref{rem:imp}).

Since~$\rho$ has infinite image it belongs, up to conjugation, either to the group of translations~$\{z\mapsto z+\beta\}$ or to the group of homotheties~$\{z\mapsto \alpha z\}$. We will deal separately with these cases.

\subsubsection*{Case 1}
We begin with the case where the image of~$\rho$ belongs to the group of translations. The action of~$\lambda\in\Lambda$ is given by
\begin{equation}\label{equivariant-vf} (z,t)\mapsto(z+\rho(\lambda),t+\lambda).\end{equation}
Under the monodromy, all the points accumulate to the point at infinity. Hence, all the leaves of~$\mathcal{G}$ other than~$\{z=\infty\}$ accumulate to this one. This implies that~$a_i$ is holomorphic, for otherwise (remark~\ref{otherpoles}) there is a component~$P$ of the curve of zeros that intersects~$\{z=\infty\}$, and any other integral curve of~$\mathcal{G}$  would intersect~$P$ infinitely many times as it accumulates to~$\{z=\infty\}$, acquiring infinitely many zeros in the process.

Let us prove that~$b_i$ is holomorphic as well, starting with the case where the ramified cover~$r$ is trivial. If the curve of poles of~$\widetilde{V}$ intersects~$\{z=\infty\}$ then~$\widetilde{L}$, which accumulates to~$\{z=\infty\}$, intersects the curve of poles, but this would contradict lemma~\ref{fewstat}. In this case, $b_i$ is holomorphic for every~$i$. Let us establish this fact in the case where the ramified cover~$r$ is not trivial. As before, if~$b_i$ were not holomorphic for every~$i$, $\widetilde{L}$ would intersect infinitely many times a  component of  the curve of poles of~$\widetilde{V}$ which is not a fiber (remark~\ref{otherpoles}). This component must be pointwise fixed by~$\sigma$ ($r$ should ramify along it), in order for the vector field induced on the quotient of~$\widetilde{L}$ under~$\sigma$ to be holomorphic. If~$\sigma$ fixes a point, it also fixes its fiber, and does not have fixed points close to the preserved fiber which are not on it: all the components of the curve of poles of~$\widetilde{V}$ which intersect~$\{z=\infty\}$ are fibers. This contradiction shows that~$b_i$ is holomorphic.

The condition for~$\widetilde{V}$ to be invariant is that its coefficient~$H$ is invariant under~(\ref{equivariant-vf}), that for every~$\lambda\in\Lambda$,
\[c(t)\frac{z^n+\cdots}{z^m+\cdots}=c(t+\lambda)\frac{(z+\rho(\lambda))^n+\cdots}{(z+\rho(\lambda))^m+\cdots}=c(t+\lambda)\frac{z^n+\cdots}{z^m+\cdots},\]
which implies that~$c$ is elliptic, and, since it has no zeros, constant. If~$m=0$, $\widetilde{V}$ is holomorphic, a case we had already discarded, so we may suppose that~$m\neq 0$. Both the numerator and the denominator of~$H$ must be invariant under~(\ref{equivariant-vf}). For the denominator,  
\begin{eqnarray*}
	z^m+b_{m-1}(t)z^{m-1}+\cdots & = & (z+\rho(\lambda))^m+b_{m-1}(t+\lambda)(z+\rho(\lambda))^{m-1}+\cdots \\ & =  & z^m+(m\rho(\lambda)+b_{m-1}(t+\lambda))z^{m-1}+\cdots. 
\end{eqnarray*}
Hence, $b_{m-1}(t+\lambda)=b_{m-1}(t)-m\rho(\lambda)$. Differentiating,
$b_{m-1}'(t+\lambda)=b_{m-1}'(t)$. Thus, $b_{m-1}'$ is elliptic and, since it is holomorphic, constant, and~$b_{m-1}(z)=\alpha z+\beta$, for some~$\alpha,\beta\in\mathbf{C}$. This implies that~$\rho(\lambda)=-(\alpha/m)\lambda$ (since~$\rho$ has infinite image, $\alpha\neq 0$). In particular,  $\rho$  is actually induced by a $\mathbf{C}$-linear map. Up to a linear change of coordinates in~$z$, we may suppose that~$\rho(\lambda)=\lambda$.  We have, for~$\zeta(z,t)=z-t$,
\[\widetilde{V}=c_0\frac{\zeta^n+\widehat{a}_{n-1}(t)\zeta^{n-1}+\cdots+\widehat{a}_0(t)}{\zeta^m+\widehat{b}_{m-1}(t)\zeta^{m-1}+\cdots+\widehat{b}_0(t)}\del{t}=\widehat{H}(\zeta,t)\del{t},\]
for some meromorphic functions~$\widehat{a}_i$ and~$\widehat{b}_i$. In this form, since~$\zeta$ is invariant by~(\ref{equivariant-vf}), the condition for the invariance of~$\widetilde{V}$ is that~$\widehat{a}_i$ and~$\widehat{b}_i$ are elliptic functions for every~$i$. However, our previous reasoning implies that they do not have poles, that they are constant,  and thus~$\widetilde{V}=\widehat{H}(\zeta)\indel{t}$.

In the case where the ramified cover~$r$ is trivial, this is sufficient to conclude: the mapping onto~$\mathbf{P}^1$ given by~$(z,w)\mapsto \zeta$ is invariant under the action of~$\Lambda$ and is well-defined on~$M$, where it gives an elliptic fibration. It maps~$V$ to the vector field~$- \widehat{H}(\zeta)\indel{\zeta}$ and  the elliptic fibration is preserved. In the case where the ramified cover~$r$ is not trivial,  suppose that~$\sigma$ preserves the fiber~$t=0$, that it is of the form~$(z,t)\mapsto (g(t,z),\mu t)$ with~$\mu ^p=1$ ($\mu \neq 1$, $p\in\{2,3,4,6\}$).  Since~$\sigma$ must preserve~$\{z=\infty\}$, it is, up to a translation in~$z$, of the form~$(z,t)\mapsto(\nu z, \mu  t)$. Since~$\sigma$ preserves~$\widetilde{V}$, $\sigma$ normalizes the transformations of the suspension. The conjugation of~(\ref{equivariant-vf}) under~$\sigma$ is~$(z,t)\mapsto(z+\nu\lambda,t+\mu \lambda)$, and since this transformation must belong to the group of the suspension for every~$\lambda\in\Lambda$, $\nu=\mu $. Since~$\mu \widehat{H}(\zeta)=\widehat{H}(\mu \zeta)$, then, by L\"uroth's theorem, $\widehat{H}(\zeta)=-\zeta R (\zeta^p)$ for some rational function~$R$.  The projection onto~$\mathbf{P}^1$ given by~$\xi=\zeta^p$  maps~$\widetilde{V}$ to the vector field~$p\xi R(\xi)\indel{\xi}$. It is well-defined in~$M$ and in its quotient under the cyclic automorphism and gives an elliptic fibration from~$S$ onto~$\mathbf{P}^1$ preserved by~$X$, proving the proposition in this case.   

\subsubsection*{Case 2}
We   now consider the case where~$\rho$ takes values in~$\mathbf{C}^*$. Either~$\rho$ contains a hyperbolic element or~$|\rho(\lambda)|=1$ for every~$\lambda$ and, in particular, its image, which is infinite, is not discrete. The action of~$\Lambda$ on~$\mathbf{C}\times\mathbf{P}^1$ is given by 
\begin{equation}\label{trans1}(z,t)  \mapsto (\rho(\lambda)z,t+\lambda).\end{equation}

\paragraph*{Case 2a} If there is some~$\lambda\in\Lambda$ such that~$|\rho(\lambda)|\neq 1$, all the integral curves of~$\mathcal{G}$ other than~$\{z=0\}$ and~$\{z=\infty\}$ accumulate to~$\{z=\infty\}$. The arguments of case~1 allow to conclude that~$a_i$ and~$b_i$ are holomorphic for every~$i$.

In order for~$\widetilde{V}$ to be invariant under~(\ref{trans1}), $H$ must be invariant as well and thus
\begin{multline}\label{camporaro}c(t+\lambda)\frac{\rho^n(\lambda) z^n+a_{n-1}(t+\lambda)\rho^{n-1}(\lambda)z^{n-1}+\cdots+a_0(t+\lambda)}{\rho^m(\lambda)z^m+b_{m-1}(t+\lambda)\rho^{m-1}(\lambda)z^{m-1}+\cdots+b_0(t+\lambda)}= \\ =c(t)\frac{z^n+a_{n-1}(t)z^{n-1}+\cdots+a_0(t)}{z^m+b_{m-1}(t)z^{m-1}+\cdots+b_0(t)}.\end{multline}
In particular, $c(z+\lambda)=\rho(\lambda)^{m-n}c(z)$. This implies that~$c'/c$ is an elliptic  function. Its poles correspond to the poles of~$c$, for~$c$ has no zeros;  since at each pole the residue is minus the order of the corresponding pole of~$c$, this residue is strictly negative. Since the sum of the residues of an elliptic function over the poles belonging to a fundamental domain vanishes, we conclude that the elliptic function~$c'/c$ has actually no poles at all, that it is constant: $c(t)=c_0e^{\delta t}$ for some~$c_0,\delta\in\mathbf{C}$, and $\rho^{m-n}(\lambda)=e^{\delta\lambda}$. From~(\ref{camporaro}), we have that $a_i(z+\lambda)=a_i(z)\rho^{n-i}(\lambda)$. This implies that either~$a_i\equiv 0$ or that~$a_i'/a_i$ is elliptic. In the second case, the poles of the elliptic function~$a_i'/a_i$ correspond to the zeros of the holomorphic function~$a_i$ and all of its  residues are positive, which implies that it is constant. Hence, $a_i(z)=A_i e^{\alpha_i z}$ for some~$\alpha_i,A_i$ in~$\mathbf{C}$ and~$\rho^{n-i}(\lambda)=e^{\alpha_i \lambda}$. Likewise, $b_i(z)=B_i e^{\beta_i z}$ and~$\rho^{m-i}(\lambda)=e^{\beta_i \lambda}$. The known powers of~$\rho(\lambda)$  are 
\[J=\{n-i|a_i\not\equiv 0\}\cup \{m-i| b_i\not\equiv 0\} \cup(\{m-n\}\setminus\{0\}).\]
If~$J=\emptyset$, $H(z,t)=c(t)\indel{t}$, and the invariance condition implies that~$c$ is elliptic and, since it does not have zeros, constant. But this implies that~$\widetilde{V}$ is holomorphic, a case we had already eliminated from  our study. Suppose thus that~$J\neq\emptyset$ and let~$d$ be the greatest common divisor of the elements of~$J$. Since~$d=\sum_{j\in J}l(j) j$ for some~$l(j)\in\mathbf{Z}$ and since we know that for every~$j\in J$ there is some~$\gamma(j)\in\mathbf{C}$ such that~$\rho^j(\lambda)=e^{\gamma(j) \lambda}$, for~$\kappa=d^{-1}\sum_{j\in J}l(j) \gamma(j)$, 
\[\rho^d(\lambda)=\rho^{\sum l(j) j}(\lambda)=e^{\lambda \sum l(j) \gamma(j)}=e^{\kappa d\lambda}.\]
In this way, for every~$\lambda\in\Lambda$ we have that~$e^{\alpha_i \lambda}=e^{\kappa(n-i)\lambda}$ and, in particular, for~$\lambda\in\Lambda\setminus\{0\}$, $\alpha_i\equiv\kappa(n-i) \mod 2i\pi/\lambda$. Since~$\Lambda$ is a lattice and contains two $\mathbf{R}$-linearly independent elements, $\alpha_i=\kappa(n-i)$. Similarly, $\beta_i=\kappa(m-i)$ and~$\delta=\kappa(m-n)$. Thus,
\[H(z,t)=c_0\frac{e^{-\kappa n  t}z^n+A_{n-1}e^{-\kappa(n-1)  t}z^{n-1}+\cdots+A_0}{e^{-\kappa m t}z^m+B_{m-1}e^{-\kappa(m-1)  t}z^{m-1}+\cdots+B_0},\]
this is, $H$ is a rational function of~$\zeta=ze^{-\kappa t}$, $H=R(\zeta)$.

This is enough to conclude in the case where the ramified cover~$r$ is trivial. The projection onto~$\mathbf{P}^1$ given by~$(z,t)\mapsto \zeta$ is invariant under~(\ref{trans1}) and maps~$\widetilde{V}$ to~$-\kappa\zeta R(\zeta)\indel{\zeta}$. The fibration, which descends to~$M$, is an elliptic one and is preserved by the vector field. In the case were the cyclic group of symmetries is not trivial, the holomorphic symmetries of~$(M,\mathcal{G})$ are induced in~$\mathbf{C}\times\mathbf{P}^1$ by transformations that are either of the form~$(z,t)\mapsto (\nu z,\mu  t)$,  or of the form~$(z,t)\mapsto (1/z,\mu  t)$, $\mu\neq 1$. In the first case, if~$\widetilde{V}$ is preserved, $\mu  R(ze^{-\kappa t})= R(\nu ze^{-\kappa \mu  t})$, but this is impossible, for it implies that~$\mu = 1$. In the second case, $\mu  R(ze^{-\kappa t})= R(z^{-1}e^{-\kappa \mu  t})$, and thus~$\mu =-1$;   the invariance condition on~$R$ becomes~$R(\zeta)=-R(1/\zeta)$ and thus, by L\"uroth's theorem, $R(\zeta)=(\zeta-1/\zeta)Q(\zeta+1/\zeta)$ for some rational function~$Q$. On~$S$, the mapping induced by~$\xi(\zeta)= \zeta+1/\zeta$ is well-defined and induces an elliptic fibration preserved by~$X$.  
\paragraph*{Case 2b}  If~$|\rho(\lambda)|= 1$ for every~$\lambda$ then the function~$\psi(z)=\log(|z|)$ on~$\mathbf{C}\times\mathbf{C}^*$ gives a well-defined function on the subset~$M_0$ of~$M$ which is the complement of the elliptic curves~$E_0$ and~$E_\infty$ coming from~$\{z=0\}$ and~$\{z=\infty\}$. Let~$P\subset M$ be an irreducible component of the curve of zeros or poles of~$V$ that is not contained in~$E_0$ or~$E_\infty$ and let $P_0=P\cap M_0$. The restriction of~$\psi$ to~$P_0$ is harmonic. It is not constant, for this would imply that the closure of the algebraic curve~$P$ is a level threefold of~$\psi$, and takes thus all real values. But this implies that~$P_0$  intersects all leaves of~$\mathcal{G}|_{M_0}$, in particular~$L$, infinitely many times. The curve~$P$ cannot be a curve of zeros, for this  would contradict lemma~\ref{fewstat}; if it is a curve of poles, it must, by the previous arguments, belong to a fiber. We must conclude that~$V=c(t)z^{s}\indel{t}$, $s\in\mathbf{Z}$. The arguments in case~2a imply that~$c(t)=c_0e^{\delta t}$ and that~$\rho^s(\lambda)=e^{\delta\lambda}$. But since~$\Lambda$ is a lattice and contains two~$\mathbf{R}$-linearly independent elements and~$|\rho(\lambda)|=1$ for every~$\lambda\in\Lambda$, $s=0$. In this case, $c$ is elliptic without zeros, thus constant, and~$\widetilde{V}$ is holomorphic.\\

This finishes the proof of proposition~\ref{prop:susp}, which completes the proof of the Main Theorem in the case where there is a univalent maximal solution that accumulates the locus of poles in a non-analytic way (in particular, in the case where a solution cannot be extended  to a map defined in~$\mathbf{C}$).

\begin{remark}
	Following the arguments in~\cite{nk}, involving classification results by Kodaira and Enoki, one may show from theorem~\ref{thamabis} that, under the hypothesis that there is a univalent maximal solution that cannot be extended  to all of~$\mathbf{C}$, the conclusion of the Main Theorem is also valid in the non-K\"ahler setting.
\end{remark}

\section{Entire solutions}\label{sec:entire}
Let~$S$ be a complex compact algebraic surface, $X$ a meromorphic vector field on~$S$ and~$\phi:\mathbf{C}\to S$ a map which is almost everywhere a solution to~$X$ and whose image is not contained in any proper analytic subset. We will further suppose that all the solutions of~$X$ in~$\mathfrak{U}$, the uniformizable locus, can be extended to entire maps. It is in this setting that we will prove the Main Theorem. McQuillan showed that foliations on algebraic surfaces admitting tangent entire curves are very special:

\begin{theorem}[McQuillan~\cite{mcquillan-ncmt}, \cite{mcquillan-canonical}]\label{mq} Let~$S$ be an algebraic surface and~$\mathcal{F}$ a holomorphic foliation on~$S$. Let~$\phi:\mathbf{C}\to S$ be a map tangent to the foliation having Zariski-dense image. Then the foliation is either Riccati or turbulent, or, up to a birational transformation, $\mathcal{F}$ is induced by a holomorphic vector field on a ramified cover of~$S$.
\end{theorem}

Our proof of the Main Theorem in the present setting is based on this result and will follow case-by-case considerations according to it. It will be spread throughout sections~\ref{mqcov}, \ref{sec:thric}, and~\ref{sec:turb}.  In all of them, we will suppose that~$S$ is a surface, $X$ a meromorphic vector field on~$S$ and that~$\phi:\mathbf{C}\to S$ is like in the beginning of this section. Let~$\mathcal{L}=\phi(\mathbf{C})$.

\subsection{Covered by a vector field}\label{mqcov} The last possibility in theorem~\ref{mq} is that, up to a birational map on~$S$, there is a finite ramified covering~$r:M\to S$ of a compact surface~$M$ where the pull-back of~$\mathcal{F}$ is generated by a holomorphic (hence complete) vector field~$V$ with isolated singularities.   Up to birational maps, the ramified covering is Galois~\cite[fact~IV.3.3]{mcquillan-canonical}, and acts by preserving the foliation induced by~$V$ (but not necessarily~$V$ itself). From the classification of holomorphic vector fields on algebraic surfaces (see~\cite[ch.~6, props.~5 and~6]{brunella-birational}) we can extract the following result:

\begin{proposition}\label{vfkah} A holomorphic vector field on a compact complex projective surface that does not have a first integral is, up to a birational transformation, either
	\begin{itemize}
		\item a constant vector field on an Abelian surface,
		\item a product vector field on~$\mathbf{P}^1\times\mathbf{P}^1$, or
		\item a vector field tangent to a Riccati foliation obtained by suspension over an elliptic curve.
	\end{itemize}
\end{proposition}

In the third case the Main Theorem follows from proposition~\ref{prop:susp}. We will now prove the Main Theorem for meromorphic vector fields covered by foliations associated to the first two cases.

\subsubsection{Constant vector fields on Abelian surfaces} We  will first describe the intersections of algebraic curves with the orbits of these vector fields.  

\begin{lemma}\label{abint} Let~$M$ be an Abelian surface, $\mathcal{F}$ a foliation given by a constant vector field~$V$, $E\subset M$ an algebraic curve which is not invariant by~$\mathcal{F}$, $L$ a leaf of~$\mathcal{F}$. Either~$L$ is closed or it intersects~$E$ infinitely many times.
\end{lemma}
\begin{proof} Let~$\eta$ be a holomorphic one-form in~$M$ such that~$\eta|_V \equiv  0$. The integral of~$\eta$ in~$M$ measures distances between leaves of~$\mathcal{F}$: if~$\rho:[0,1]\to M$ is such that~$\int_\rho\eta=0$, $\rho(0)$ and~$\rho(1)$ belong to the same leaf of~$\mathcal{F}$. Let~$\gamma:[0,1]\to M$ be a path joining a point~$p\in E$ to a point in~$L$. The (ramified) translation structure induced by~$\eta$ on~$E$ is geodesically complete. Thus, there is a curve~$\widehat{\gamma}:[0,1]\to E$, $\widehat{\gamma}(0)=p$, such that~$\gamma^*\eta\equiv\widehat{\gamma}^*\eta$. Hence, $\widehat{\gamma}(1)$ belongs to~$E\cap L$, which is, in particular, not empty. Let~$q$ be a point in this intersection.  We may endow~$M$ with the structure of a real Abelian Lie group having its zero at~$q$. Since~$L$ is a subgroup of~$M$, its closure~$\overline{L}$ is a real (closed) Lie subgroup of~$M$, a torus. Its real dimension is either two (when~$L$ is closed), three, or four (when~$L$ is dense in~$M$). There is nothing to prove if~$L$ is closed or if~$L$ is dense in~$M$, so we need only consider the case where~$\overline{L}$ is a three-dimensional real torus. Let~$h:U\to \mathbf{C}$ be a local primitive of~$\eta$ defined in a neighborhood of~$q$ induced by an affine function in~$\mathbf{C}^2$ such that~$\overline{L}$ is locally given by~$\Im(h)=0$ (this is a particular instance of Cartan's theorem on the local structure of real-analytic, non-singular, Levi-flat submanifolds~\cite[Th\'eor\`eme~IV, p.~23]{cartan}). In a neighborhood of~$q$ within~$E$, the zero level set of the harmonic function~$\Im(h|_E)$, $\overline{L}\cap E$, is infinite. Since~$L$ is dense in~$\overline{L}$, $L$ intersects~$E$ infinitely many times. \end{proof}

Let us prove the Main Theorem for meromorphic vector fields covered by foliations associated to the first possibility in proposition~\ref{vfkah}. Let~$S$ be a compact surface, $X$ a meromorphic vector field on~$S$, and~$\mathcal{L}$ like in the beginning of this section. Let~$M$ be an Abelian surface, $V$ a constant vector field on~$M$, $r:M\to S$ a ramified cover and~$R$ a meromorphic function on~$M$ such that~$r^*X=RV$.

Let~$\widetilde{\mathcal{L}}$ be the orbit of~$V$ that projects onto~$\mathcal{L}$ (we may suppose that~$\widetilde{\mathcal{L}}$ is not closed for otherwise~$\mathcal{L}$ would be closed as well). Let~$E$ be an irreducible component of the curve of zeros of~$R$. Since~$V$ has no closed orbit, $E$ is generically transverse to~$\mathcal{F}$. By lemma~\ref{abint}, $\widetilde{\mathcal{L}}$ intersects~$E$ infinitely many times, and the restriction of~$RV$ to~$\widetilde{\mathcal{L}}$ has infinitely many zeros. By lemma~\ref{fewstat}, the solution of~$r^*X$ taking values in~$\widetilde{\mathcal{L}}$ has infinitely many determinations and so must the solution of~$X$ taking values in~$\mathcal{L}$. We conclude that~$R$ has no zeros, and, in consequence, no poles either, that it is thus constant, and that~$RV$ is holomorphic. Since the Galois group of~$r$ acts by preserving~$RV$, $X$ is holomorphic as well.

\subsubsection{Product vector fields}\label{sub:product} Product vector fields on~$\mathbf{P}^1\times\mathbf{P}^1$ that do not admit a first integral are, in suitable coordinates and up to a constant factor, of the form~$z\indel{z}+\indel{w}$ or~$x\indel{x}+\mu y\indel{y}$, $\mu\notin\mathbf{Q}$. Multiplying these by rational functions often produces vector fields whose solutions have infinitely many determinations:

\begin{lemma}\label{cxp1} If a meromorphic vector field on~$\mathbf{P}^1\times\mathbf{P}^1$ admits a transcendental solution with finitely many determinations and is of the form
	\begin{itemize}
		\item $R(z,w)(z\indel{z}+\indel{w})$, then $R$ is a function of~$w$; 
		\item $R(z,w)(z\indel{z}+\mu w\indel{w})$, with $\mu\notin\mathbf{Q}$, then~$R$ is constant.
	\end{itemize}
\end{lemma}

\begin{proof} In the first case, consider a transcendental leaf parametrized by~$s\mapsto (z_0e^s,s)$, $s\in\mathbf{C}$. The induced vector field is~$R(z_0e^s,s)\indel{s}$. Since its solutions have finitely many determinations, from proposition~\ref{vf-essential}, this vector field is a rational one and~$R$ is a rational function of~$w$ (is independent of~$z$). In the second, parametrize an orbit of by~$s\mapsto(x_0e^s,y_0e^{\mu s})$, $s\in\mathbf{C}$. The induced vector field is~$R(x_0e^s,y_0e^{\mu s})\indel{s}$. By proposition~\ref{vf-essential}, if this solution has finitely many determinations, $R(x_0e^s,y_0e^{\mu s})$ is a rational function of~$s$. But since~$\mu\notin \mathbf{Q}$, $R$ must be constant.\end{proof}

Let us now prove the Main Theorem for meromorphic vector fields covered by foliations associated to the second possibility in proposition~\ref{vfkah}. Let~$S$ be a compact surface, $X$ a meromorphic vector field on~$S$, and~$\mathcal{L}$ like in the beginning of this section, $r:\mathbf{P}^1\times\mathbf{P}^1\to S$ a ramified covering, $V$ a holomorphic  product vector field on~$\mathbf{P}^1\times\mathbf{P}^1$ such that~$r^*X=RV$. The Galois group of~$r$ acts by preserving~$RV$. A description of the birational automorphisms of the associated foliation on~$\mathbf{P}^1\times\mathbf{P}^1$, from which we will borrow some arguments, appears in~\cite{jvp-height}.

First consider the case where~$V=z\indel{z}+\indel{w}$. By lemma~\ref{cxp1}, $r^{*}X=R(w)(z\indel{z}+\indel{w})$. In order to understand the group of birational automorphisms preserving~$RV$, let us begin by describing the larger group of birational transformations preserving the foliation induced by~$V$. We claim that the latter is isomorphic to~$\mathbf{Z}/2\mathbf{Z}\ltimes(\mathbf{C}\times\mathbf{C}^*)$, where~$\mathbf{C}$ acts via the flow of~$\indel{w}$, $\mathbf{C}^*$ via that of~$z\indel{z}$ and~$\mathbf{Z}/2\mathbf{Z}$ by
\begin{equation}\label{ordertwo}(z,w)\mapsto(1/z,-w).\end{equation}
Since the transcendental leaves are Zariski-dense, an automorphism is determined by its action upon such a leaf. The group $\mathbf{C}\times\mathbf{C}^*$ acts transitively upon the set of transcendental pointed leaves. If an automorphism acts upon the leaf parametrized by~$s\mapsto(e^s,s)$, $s\in\mathbf{C}$, fixing the point~$s=0$, it is of the form~$(e^{a s}, a s)$, $a\in\mathbf{C}$, since the only biholomorphisms of~$\mathbf{C}$ fixing~$0$ are those of the form~$s\mapsto as$. If such an automorphism extends to~$\mathbf{C}^2$, it extends as~$(z,w)\mapsto (z^a,aw)$, whose inverse is~$(z,w)\mapsto (z^{1/a},w/a)$. In order for the automorphism and its inverse to be well-defined, both~$a$ and~$a^{-1}$ must be integers, and thus $a^2=1$, proving our claim. The projection onto~$w$ is equivariant with respect to the action of this group of order two, and invariant under the action of~$\mathbf{C}^*$. The group acting on~$w$ is~$\mathbf{Z}/2\mathbf{Z}\ltimes\mathbf{C}$. All finite subgroups of the latter are either trivial or have order two and are generated by a map of the form~$w\mapsto -w+c$ (in which case we may suppose that~$c=0$). If~$R(w)(z\indel{z}+\indel{w})$ is invariant under the action of some group acting trivially on~$w$, the projection onto~$w$ maps the vector field onto~$R(w)\indel{w}$, and the fibration is preserved. If it is invariant under the action of some group acting on~$w$ by~$w\mapsto -w$, then~$R(-w)=-R(w)$ and thus~$R(w)=wQ(w^2)$ for some rational function~$Q$. The image of the vector field under the quotient map~$(z,w)\mapsto w^2$ is, for~$\xi=w^2$, $2\xi Q(\xi)\indel{\xi}$. The vector field preserves this fibration.

Consider now the case where~$V=x\indel{x}+\mu y\indel{y}$, $\mu\notin\mathbf{Q}$. 
By lemma~\ref{cxp1}, $r^*X$ is a constant multiple of~$V$ and is thus holomorphic. The Galois group of~$r$ acts by preserving this vector field and thus~$X$ is holomorphic as well. 

We thus prove the Main Theorem in these cases.

\subsection{Riccati}\label{sec:thric}  Let~$\Pi:S\to B$ be a rational fibration and~$X$ a meromorphic vector field on~$S$ inducing a Riccati  foliation~$\mathcal{F}$ adapted to~$\Pi$. We will suppose that the Riccati foliation is in standard form, as discussed in section~\ref{subsec-ric}. Let~$\phi:\mathbf{C}\to S$ be a map with Zariski-dense image that is almost everywhere a solution to~$X$ and let~$\mathcal{L}=\phi(\mathbf{C})$. Since we have the holomorphic function~$\Pi\circ\phi:\mathbf{C}\to B$, $B$ is either rational or elliptic.

The vector field may preserve the rational fibration~$\Pi$. For this to happen, it is sufficient for~$X$ to be holomorphic in the neighborhood of a generic fiber. In this situation, the local flow will map fibers to fibers (by the maximum principle) and will induce a non-identically-zero holomorphic vector field~$Y$ on~$B$. By lemma~\ref{fewstat}, $Y$ will extend holomorphically to~$B$, proving the Main Theorem for this case. However, this is not the only possibility, and there will be cases where the vector field will not preserve the rational fibration. In this case, we will prove the Main Theorem either by showing that the vector field is birationally equivalent to a holomorphic one, or by exhibiting another fibration that is invariant under the flow (like in proposition~\ref{prop:susp}). Example~\ref{ex:notriccati1} gives an example of this last situation: a vector field with a transcendental univalent solution inducing a foliation which is Riccati with respect to a rational fibration, but where this fibration is not preserved by the flow (in this example, another fibration, elliptic, is preserved).

For the proof of the Main Theorem we now begin a series of cases that will cover all the possibilities. Let~$\rho:\pi_1(B\setminus\{p_1, \ldots, p_k\})\to \mathrm{PSL}(2,\mathbf{C})$ be the monodromy representation,  $\Gamma$  its image.

\subsubsection{When the orbit of the monodromy containing the solution is finite}\label{ss:finorb} 
This is the case, for instance, when the monodromy of the Riccati equation is finite.

\begin{proposition}\label{notfinite}  Let~$X$ be a vector field on the surface~$S$ inducing a Riccati foliation adapted to the fibration~$\Pi:S\to B$. Let~$H$ be a generic fiber, let~$\phi:\mathbf{C}\to S$ be almost everywhere a solution to~$X$, and let~$\mathcal{L}$ be its image. If~$\mathcal{L}$ is Zariski-dense and~$\mathcal{L}\cap H$ is finite, $X$ preserves~$\Pi$. 
\end{proposition}

\begin{proof} We begin with the cases where~$B=\mathbf{P}^1$. If~$\phi$ does not have periods, let~$\psi:\mathbf{C}\to\mathbf{P}^1$ be given by $\psi=\Pi\circ\phi$. Our assumption is that, for a generic~$t_0$, 
	$\{ t\in\mathbf{C}|\psi(t)=\psi(t_0)\}$ is finite. This implies that $\psi$ is a rational function of~$t$ (for otherwise it would have an essential singularity at~$t=\infty$ and, by Picard's Great Theorem, our finiteness assumption would be violated). Under~$t\mapsto\psi(t)$, $X$ induces a rational multivalued vector field~$V$ in~$\mathbf{P}^1$.  The set of points in~$S$ such that~$\Pi$ maps~$X$ onto any one of the determinations of~$V$ is a closed analytic subset of~$S$ which contains the Zariski-dense curve~$\mathcal{L}$ and which is thus all of~$S$. Along a fiber of~$\Pi$, the projection of~$X$ onto~$B$ is locally constant and is thus constant. We conclude that~$V$ is a honest rational vector field and that the fibration is preserved  (moreover, since~$V$ has the single-valued solution~$\psi$, the vector field on~$B$ is holomorphic and, since~$\psi$ is rational, it reads~$\indel{z}$ in a suitable coordinate). Suppose now that~$\phi$ has a period and, without loss of generality, that this period is~$2i\pi$. Let~$\zeta=e^t$, so that~$\phi(t)=\sigma(\zeta)$ for some~$\sigma:\mathbf{C}^*\to S$. Let~$\psi:\mathbf{C}^*\to B$ be given by~$\Pi\circ\sigma$. The finiteness assumption is that for a generic~$\zeta_0$, $\{ \zeta\in\mathbf{C}^*|\psi(\zeta)=\psi(\zeta_0)\}$ is finite. This implies that $\psi$ is a rational function of~$\zeta$ (for otherwise it would have an essential singularity at~$\zeta=0$ or~$\zeta=\infty$). The image of~$\zeta\indel{\zeta}$ under~$\psi$
	is a rational multivalued vector field~$Y$ on~$B$. A variation of the previous arguments shows that~$B$ is rational and that~$Y$ is  holomorphic (and that in a convenient coordinate reads~$z\indel{z}$). This proves that the fibration is preserved.

	The case where~$B$ is elliptic does not arrive. Suppose that~$B$ is an elliptic curve and that~$\phi$ has no periods and consider the function $\psi:\mathbf{C}\to B$ as before. Let~$\theta:B\to\mathbf{P}^1$ be an auxiliary nonconstant rational map. Since~$\psi$ has only finitely many points in the preimage of any given one, so does~$\theta\circ\psi$. By Picard's Great Theorem, $\theta\circ\psi$ must extend holomorphically to~$\infty$, and so must~$\psi$. This yields a nonconstant function~$\overline{\psi}:\mathbf{P}^1\to B$, which is impossible. If~$\phi$ has periods, we may suppose without loss of generality that~$2i\pi$ is a primitive period, that~$\phi(t)=\sigma(e^t)$ for some rational function~$\sigma$. The function~$\psi:\mathbf{C}^*\to B$ defined as before must extend holomorphically to~$0$ and~$\infty$. Again, this is impossible.\end{proof}

\subsubsection{When the monodromy is discrete, infinite and elementary}\label{con:discele} In this case the monodromy of the Riccati equation either
\begin{enumerate} 
	\item \label{en:af} has a subgroup of finite index consisting of translations,  or
	\item \label{en:nonaf} is contained in the group that permutes~$\{0,\infty\}$ and contains a subgroup of~$\mathbf{C}^*$ with index at most two~\cite[sections~59--62]{ford}. 
\end{enumerate}
In the first case there is one fixed point for~$\Gamma$; in the second, two fixed points or an orbit of size two. We may suppose, by the results in section~\ref{ss:finorb}, that~$\mathcal{L}$ does not correspond to one of these finite orbits, and it thus accumulates to each one of them, which belong  to~$\mathfrak{U}$.   Let~$\Sigma$ be a leaf of~$\mathcal{F}$ corresponding to a finite orbit of~$\Gamma$. If~$\Sigma$ is Zariski-dense, $X$ would be holomorphic and nonzero at a generic point of~$\Sigma$ and we would be in the setting of  proposition~\ref{notfinite}. We may thus suppose that~$\Sigma$ is an irreducible algebraic curve.

The present setting limits the possible kinds of special fibers (section~\ref{subsec-ric}) appearing in the Riccati foliation.

\begin{proposition}\label{sf:ndpod} If the monodromy is discrete, infinite and elementary, all the special fibers are either nondegenerate parabolic  or dicritical. 
\end{proposition}

\begin{proof}  If every  irreducible component of the curve of poles is contained in a fiber, $X$ preserves the fibration. Otherwise, there is an irreducible component of the curve of poles generically transverse to~$\Pi$ and, by~(\ref{trzer}), an irreducible component  of the curve of poles that does so as well.

	Suppose that~$\Sigma$, a leaf of~$\mathcal{F}$ associated to a finite orbit of~$\Gamma$, is a component of the curve of zeros or poles of~$X$. Let us discard the existence of semidegenerate and nilpotent special fibers (for these,  the only invariant curves that are not contained in the fiber are, after desingularization, weak separatrices of saddle nodes). If we had such a special fiber then~$\Sigma$ would be (after reducing the singularities of~$\mathcal{F}$ if need be) the weak separatrix of a saddle node. Up to further blowups, we may suppose that this saddle-node is a point where~$X$ is reduced in the sense of section~\ref{bir}. Since~$\Sigma$ is accumulated by every orbit, it belongs to~$\mathfrak{U}$, and the results of section~\ref{sec:saddle-node} apply: $X$ is holomorphic and nonzero along~$\Sigma$  (the weak separatrix of the saddle node), contradicting the hypothesis made on~$\Sigma$. We conclude that no special fiber of the Riccati foliation may be semidegenerate or nilpotent. In an analogous way we discard the existence of nondegenerate nonparabolic special fibers. In this   case, $\Sigma$ would be a separatrix through one of the two nodes that~$\mathcal{F}$ has at such a special fiber. According to the results of section~\ref{poincare}, $X$ is holomorphic and has an isolated equilibrium point at this node, but this contradicts the fact that~$X$ has zeros or poles along~$\Sigma$, proving that the special fibers cannot be of this kind.   
	
	Let us thus suppose that no irreducible component of the locus of zeros or poles is invariant by~$\mathcal{F}$. Let~$E$ be an irreducible component of the curve of zeros of~$X$ (it is generically transverse to~$\mathcal{F}$). Let~$(S',F')\to(S,F)$ be a reduction of the singularities of the vector field~$X$ (as discussed in section~\ref{bir}) in the neighborhood of the special fiber~$F$, denote by~$X'$, $\mathcal{F}'$, the vector field and foliation on~$S'$ and by~$\Sigma'$ the lift of~$\Sigma$ and so on. Since~$X'$ is reduced, $E'$ intersects~$F'$ transversely at a regular point of~$\mathcal{F}'$.   If~$F$ is a semidegenerate special fiber, $\mathcal{F}$ has two saddle nodes at~$F$, whose strong separatrices are contained in~$F$. The holonomy of~$\mathcal{F}$ along its leaf~$F$ is generated by a map which is tangent to the identity (but which is different from it). The set~$\mathfrak{U}$, being closed and intersecting infinitely many times each fiber, cannot be contained in the weak separatrices of these saddle-nodes and, by~\cite[lemma~19]{camacho-rosas}, $\mathfrak{U}$ accumulates to~$F$ away from the singularities of~$\mathcal{F}$.  The curve~$F'$ is a tree of rational curves. Each one of the two saddle-nodes of~$\mathcal{F}$ is found at one of the irreducible components of~$F'$ and, at the intersection of two irreducible components of~$F'$, $\mathcal{F}'$ has a saddle. Saddles have the following property: \emph{if a sequence of leaves accumulates to one separatrix, it accumulates to the other separatrix as well}~\cite[lemma~20]{camacho-rosas}. This makes the dynamics of the holonomy of~$\mathcal{F}'$ along the strong separatrices of the saddle-nodes propagate all along~$F'$; these dynamics may be seen at the transversal~$E'$: as they approach~$F'$,  the leaves of~$\mathfrak{U}$ will intersect~$E'$  infinitely many times and, as they do so, they gain infinitely many zeros of~$X'$. This contradiction rules out the existence of semidegenerate special fibers. If~$F$ is a nilpotent special fiber, the resolution of~$X$ will have one saddle-node whose strong separatrix is contained in~$F'$ and saddles at the intersection of two irreducible components of~$F'$; the same argument proves that this cannot happen. If~$F$ is a nondegenerate hyperbolic special fiber, the singularities of~$\mathcal{F}'$ on~$F'$ are nodes (with hyperbolic holonomy) or saddles, which have an analogous dynamical property~\cite[lemma~20]{camacho-rosas}; an argument similar to the previous one allows us to conclude that this case cannot arrive.

	Finally, nondegenerate elliptic special fibers do not appear, since their monodromy is not discrete. \end{proof}

According to this result we may, from now on, suppose that all the special fibers are either nondegenerate parabolic  or dicritical.

We will now prove the Main Theorem for those vector fields associated to Riccati foliations with discrete, infinite and elementary monodromy.

\paragraph{Case~\ref{en:af}, $\Gamma$ contains an infinite group of translations}

We begin with the case where~$\Gamma$ contains an infinite group of translations with finite index, having a fixed point corresponding to a leaf~$\Sigma$ of~$\mathcal{F}$. Each fiber of~$\Pi$ intersects only one point of~$\Sigma$, and~$\Sigma$ and~$B$ get identified by~$\Pi$ (in particular, $B$ inherits a uniformizable affine structure).  Let~$H$ be a generic fiber and let~$\gamma$ be a closed path in~$\Sigma$ based at~$\Sigma\cap H$ along which the monodromy of the affine structure of~$\Sigma$ is trivial. By the Fundamental Lemma (lemma~\ref{holmon}), the holonomy of~$\mathcal{F}$ along~$\gamma$ will fix pointwise the points of~$H\cap \mathfrak{U}$ sufficiently close to~$H\cap \Sigma$ and, since this set is infinite, the monodromy of the Riccati equation along~$\gamma$ is trivial. This implies that the fibers of~$\Pi$ corresponding to points of~$\Sigma$ where the affine structure is nonsingular are transverse fibers, for, according to proposition~\ref{sf:ndpod}, all other special fibers have nontrivial monodromy. Lemma~\ref{holmon} takes here a  global form:   there is a well-defined (abstract) map from the monodromy of the affine structure on~$\Sigma$ to the monodromy of the Riccati equation. We thus  need to study the representations of the monodromy of the affine structures on~$\Sigma$ (or more generally, following remark~\ref{orbifold}, of the orbifold fundamental group of~$\Sigma$) into the affine group which have infinite, discrete and elementary image. The only affine structures that do so are the parabolic and elliptic curves and orbifolds.

We first address the cases where~$\Sigma$ is a parabolic orbifold. (i) If the affine structure in~$\Sigma$ is that of a parabolic cylinder (two singularities with infinite ramification indices), $B$ is rational and the affine structure has two singularities. If one of the two local monodromies were of finite order, the global monodromy would be finite. Thus, the two special fibers are nondegenerate parabolic and the monodromy is parabolic. By placing the singularities of the affine structure of~$\Sigma$ above~$0$ and~$\infty$, the foliation is birationally equivalent to the foliation on~$\mathbf{P}^1\times\mathbf{P}^1$ given by~$zdw-dz$, with the rational fibration adapted to the Riccati foliation being~$\Pi(z,w)=z$ and~$\Sigma$ being the curve~$\{w=\infty\}$. By lemma~\ref{cxp1}, $X=R(w)(z\indel z+\indel{w})$ for some rational function~$R$. The rational fibration onto~$\mathbf{P}^1$ given by~$(z,w)\mapsto w$  maps~$X$ onto~$R(w)\indel{w}$, and is preserved by~$X$ (notice that this is not the original rational fibration~$\Pi$). (ii) If the affine structure makes~$B$ an orbifold~$(2,2,\infty)$, the monodromy around each one of the fibers corresponding to the two points with ramification index~$2$ has order two (otherwise one of them would be trivial and this would imply that the monodromy is globally finite) and these fibers are thus dicritical of multiplicity two. From~$\Pi:S\to B$ and the ramified (uniformizing) double cover~$\eta:B'\to B$ of~$B$ by the parabolic cylinder, we may construct the rational fibration~$\eta^{*}\Pi:S'\to B'$, having a vector field~$X'$ (which may no longer have univalent solutions, but which in the worst case will have solutions with   two determinations) inducing a Riccati foliation. In~$S'$, the fibers mapping to the dicritical ones in~$S$ become regular ones so that the situation is essentially the same as in case~(i). The previous reasoning implies that~$S'$ is~$\mathbf{P}^1\times\mathbf{P}^1$ and~$X'$ a vector field of the form~$S(w)(z\indel z+\indel{w})$. The Galois group of~$S'\to S$ is generated by~(\ref{ordertwo}). According to the results in section~\ref{sub:product}, the function in~$S$ induced by~$(z,w)\mapsto w^2$ maps~$X$ into a vector field on~$\mathbf{P}^1$, and gives a rational fibration (different from the original one) preserved by~$X$.

We now deal with the cases where~$\Sigma$ is an elliptic curve or orbifold. If~$\Sigma$ is an elliptic curve, there are no special fibers; the Riccati foliation is a suspension and falls within the scope of 
proposition~\ref{prop:susp}, which proves that there is a fibration preserved by the vector field. Suppose thus that~$B$ is an elliptic orbifold and let~$\eta:E\to B$ be the smallest  elliptic curve that covers it. The order of the monodromy around the ramification points of~$\eta$ divides the ramification order of the affine structure of~$B$, and the monodromy of the Riccati foliation is an image of the monodromy of the affine structure on~$B$. Since any representation in~$\mathrm{Aff}(\mathbf{C})$ of the holonomy of~$B$ having infinite image is faithful, so is this one, and the pull-back the Riccati foliation to the rational fibration~$\eta^*\Pi:\eta^* S\to E$ is a suspension. This case also falls  within the scope of proposition~\ref{prop:susp}.

\paragraph{Case~\ref{en:nonaf}, $\Gamma$ contains a group of homotheties}

We now address the case where~$\Gamma$ is in the group that permutes~$\{0,\infty\}$. This group has no parabolic elements and, in consequence, the special fibers of the corresponding Riccati foliations will all be dicritical and the local monodromies will have finite order.

Let us begin with the cases where~$\Gamma$ is actually a subgroup of~$\mathbf{C}^*$.   The uniformizable affine structures on curves whose monodromy admits representations with infinite image in~$\mathbf{C}^*$ are the parabolic cylinder  and those supported on elliptic curves. If one of the curves corresponding to one of the fixed points of~$\Gamma$ is   a parabolic cylinder, the Riccati foliation has dicritical special fibers (hence finite monodromy) at the fibers corresponding to the singularities of the affine structure with ramification index~$\infty$. But this implies that the global  monodromy is finite, and rules out this case.  If one of the curves corresponding to one of the fixed points of~$\Gamma$ is an elliptic curve, there are no special fibers, the Riccati foliation is a suspension over an elliptic curve and falls within the scope of proposition~\ref{prop:susp}.

We will now address the cases where~$\Gamma$ contains a subgroup of~$\mathbf{C}^*$ with index two. The subgroup of~$\mathrm{PSL}(2,\mathbf{C})$ that permutes~$\{0,\infty\}$ is the normalizer~$N(\mathbf{C}^*)$ of~$\mathbf{C}^*$. It is the group that, in a convenient coordinate, is generated by~$\mathbf{C}^*$ and by the involution~$w\stackrel{\sigma}{\longrightarrow} 1/w$. All the elements in the coset~$\sigma\mathbf{C}^*$ have order two. The center of~$N(\mathbf{C}^*)$ is generated by the involution~$w \stackrel{\tau}{\longrightarrow}  -w$. The couples in~$N(\mathbf{C}^*)$ that commute are either couples in~$\mathbf{C}^*$, couples in~$\sigma\mathbf{C}^*$ differing by~$\tau$, or couples where~$\tau$ is one of their members.   

Let~$\Sigma$ be the algebraic curve in~$S$ associated to the finite orbit~$\{0,\infty\}$ of~$\Gamma$ and consider the (possibly ramified) double cover~$\Pi|_\Sigma:\Sigma\to B$.  Let~$F$ be a special fiber. The monodromy around~$F$ either fixes or exchanges the two branches of~$\Sigma$. In the first case, $\Pi|_\Sigma$ is a regular cover at the two points of~$\Sigma$ that intersect~$F$; in the second, $\Pi|_\Sigma$ is locally a ramified cover of order two (these special fibers are dicritical of multiplicity two). The   double cover~$\Pi|_{\Sigma}:\Sigma\to B$ has~$r$ ramification points, which  correspond to the second situation just described.  By the Riemann-Hurwitz formula, $\chi(\Sigma)=2\chi(B)-r$. Since~$\Sigma$ and~$B$   are both either rational or elliptic, there are three possibilities:
\begin{itemize}
	\item $\Sigma$ and~$B$ are both elliptic, $r=0$;
	\item $\Sigma$ and~$B$ are both rational, $r=2$;
	\item $\Sigma$  is elliptic, $B$ is rational  and~$r=4$. 
\end{itemize}

In the first case, there are no special fibers, for the affine structure of~$\Sigma$ has no singularities.  
The monodromy of~$\mathcal{F}$ is a representation in~$N(\mathbf{C}^*)$ of the fundamental group of~$\Sigma$ whose image intersects the coset~$\sigma\mathbf{C}^*$. However, such Abelian subgroups of~$N(\mathbf{C}^*)$ are finite, and need not be considered.

In the second case, the monodromy would be a representation of the fundamental group of~$\mathbf{C}\setminus\{0,q_1, \ldots, q_k \}$, where (assuming~$q_i$ gives a fiber of the first type described, $0$ and~$\infty$ the two points of the second one) the monodromy of a loop around~$q_i$ is an element of finite order (depending on~$i$)  of~$\mathbf{C}^*$ and the monodromies of loops around~$0$ and~$\infty$ are both   in the coset~$\sigma\mathbf{C}^*$. Such groups are finite as well, and may be discarded in this situation.

In the last case,  the affine structure of~$\Sigma$ has no singularities, and the special fibers are only the four ones along which~$\Sigma$ ramifies. By pulling-back the Riccati foliation to the pull-back of the rational fibration along the twofold ramified cover~$\Pi|_\Sigma:\Sigma\to B$, we have a meromorphic vector field inducing a Riccati foliation without singularities on a fibration over the elliptic curve, a suspension. It falls within the scope of proposition~\ref{prop:susp}.

\subsubsection{In the presence of good orbifold coverings} 
For this part  we will borrow some arguments from~\cite{brunella-complete}. We will suppose that the monodromy~$\Gamma$ is infinite. Recall that we are assuming that all maximal univalent solutions can be extended to entire maps. Let~$B^*\subset B$ be given by the union of the transverse and of the dicritical fibers and let~$S^*=\Pi^{-1}B^*$.  Let~$F_1, \ldots, F_n$ be the dicritical fibers, let~$p_i=\Pi(F_i)$ and let~$m_i$ be the corresponding multiplicity. We will consider~$B^*$ as an orbifold by affecting~$p_i$ with the angle~$2\pi/m_i$ and the other points of~$B^*$ with angle~$2\pi$. The monodromy~$\Gamma$ is, naturally, a representation of the orbifold fundamental group of~$B^*$. Affect the leaves of~$\mathcal{F}$ in~$S^*$ with an orbifold structure, by declaring that, at a dicritical fiber of multiplicity~$m$, each one of the two special curves (the separatrices corresponding to the fixed points of the local monodromy), the angle is~$2\pi/m$, the angle being~$2\pi$ for all the other points. With respect to these structures, for a leaf~$L$ of~$\mathcal{F}|_{S^*}$, $\Pi|_L:L\to B^*$ is tautologically an orbifold map. We will only need to consider the cases where
\begin{equation}\label{hypo}\mathcal{L}\subset{S^*}  \text{ and  } \mathcal{L} \text{ has no orbifold points,}\end{equation}
and we will make this assumption from now on. The second part of the condition means that if~$\mathcal{L}$ intersects a special fiber at finite time, this special fiber is dicritical  and the point of intersection is not one of the two special points of the dicritical component.  The leaves which do not satisfy~(\ref{hypo}) are separatrices at one of the special fibers, and correspond, at most, to countably many points of the intersection of~$\mathfrak{U}$ with a generic fiber~$H$. In particular, hypothesis~(\ref{hypo}) holds whenever~$\mathfrak{U}\cap H$ is uncountable.

Recall that for~$\phi:\mathbf{C}\to S$, a map that is  almost everywhere a solution to~$X$, and  for~$\mathcal{L}=\phi(\mathbf{C})$, we have denoted by~$\widehat{\mathcal{L}}$ the curve obtained from~$\mathcal{L}$ after compactifying all of its analytic parabolic ends. Let~$\Theta$ be the module of periods of~$\phi$. If~$\Theta$ is a lattice, $\mathcal{L}$ is an elliptic curve (and is thus algebraic);  if~$\Theta$ is isomorphic to~$\mathbf{Z}$, $\mathcal{L}$ is isomorphic to~$\mathbf{C}^*$ (with the structure induced by the vector field~$z\indel{z}$) and at each point of~$\widehat{\mathcal{L}}\setminus\mathcal{L}$ the vector field has a simple zero (if~$\widehat{\mathcal{L}}\setminus\mathcal{L}$ has two points, $\widehat{\mathcal{L}}$ is a rational curve); finally,  if~$\Theta$ is trivial, $\mathcal{L}$ is isomorphic to~$\mathbf{C}$ (with the structure induced by~$\indel{z}$); there is at most one point in~$\widehat{\mathcal{L}}\setminus\mathcal{L}$, where the vector field has a double zero (if~$\widehat{\mathcal{L}}\neq\mathcal{L}$, $\widehat{\mathcal{L}}$ is a rational curve). Since~$\widehat{\mathcal{L}}$ is not algebraic, either~$\widehat{\mathcal{L}}=\mathcal{L}$ or~$\widehat{\mathcal{L}}\setminus\mathcal{L}$ is a single point, where the vector field has a simple zero. Two cases arise:

\paragraph{First case, when $(\widehat{\mathcal{L}}\setminus\mathcal{L})\cap S^*=\emptyset$}\label{par:firstcase} Either~$\widehat{\mathcal{L}}= \mathcal{L}$ or~$\widehat{\mathcal{L}}\setminus \mathcal{L}$ is a single point that is not in~$S^*$. On~$\mathcal{L}$, the tautological affine structure that inherited as a quotient of~$\mathbf{C}$ and the one induced by~$X$ coincide. The mapping~$\Pi|_{\mathcal{L}}:\mathcal{L}\to B^*$ is a (orbifold) covering map of~$B^*$ and thus $\mathcal{L}$ is covered by the (orbifold) universal covering of~$B^*$. Since~$\mathcal{L}$ is an entire curve, its universal cover is biholomorphic to~$\mathbf{C}$. Since the group of biholomorphisms of~$\mathbf{C}$ preserves (tautologically) the affine structure, this induces a uniformizable affine structure on~$B^*$.

If~$B^*$ is a torus, $\mathcal{F}$ has no singularities; it is a suspension, and falls within the scope of proposition~\ref{prop:susp}. We will thus suppose that~$B^*$ is either an elliptic orbifold or a parabolic one. In the neighborhood of a regular fiber, there are coordinates~$(z,w)$ for~$\Delta\times \mathbf{P}^1$ where~$X$ has the form~$f(z,w)\indel{z}$, with~$f$ meromorphic in~$z$ and algebraic in~$w$, and where the coordinate~$z$ in the base is an affine one (with respect to the affine structure we just defined). Every leaf~$L$ is naturally endowed with two affine structures with singularities, one induced by~$X$ and one by pulling back the one on~$B$ by~$\Pi|_L$. The difference~(\ref{affinedefect}) of these affine structures is~$-f_z/f\, dz$. It vanishes for all the values of~$w$ belonging to~$\mathcal{L}$, and thus vanishes identically: on each leaf, the fibration establishes an isomorphism of affine structures between the one obtained by pullback of the one on~$B$ and the one induced by~$X$. Up to a finite ramified cover~$\eta:B'\to B$, the affine structure on~$B'$ is induced by a vector field~$V$. Let~$\eta^*\Pi:S'\to B'$ be the ruled surface obtained by pulling back~$\Pi$ along~$\eta$; it has a vector field~$X'$ coming from the natural map~$S'\to S$. There is a meromorphic function~$h$ on~$S'$ such that~$\Pi_*(h X')=V$, which is constant along the leaves. In particular, it is constant on the Zariski-dense leaf~$\mathcal{L}$, and  is thus constant. The vector field~$S'$ projects to~$V$ via~$\Pi'$. This implies that~$X'$ is holomorphic above a generic fiber of~$\Pi'$ which, on its turn, implies that~$X$ is holomorphic above a generic fiber of~$\Pi$, establishing the result in this case.  

\paragraph{Second case, when $\widehat{\mathcal{L}}\setminus \mathcal{L}\subset S^*$ } We have that~$\widehat{\mathcal{L}}\setminus \mathcal{L}=\{p\}$, $p\in S^*$. The curve~$\widehat{\mathcal{L}}$ is entire and there is a global coordinate~$\zeta$ for~$\widehat{\mathcal{L}}$, centered at~$p$, where, up to a constant factor, $X|_{\mathcal{L}}=\zeta\indel{\zeta}$. In an open subset of~$\mathbf{C}$, a vector field of the form~$f(\zeta)\indel{\zeta}$ induces an affine structure. The difference of this affine structure and the one induced by~$\zeta$  is~$-f'/f\,d\zeta$. In particular, the affine defect of the affine structure induced by~$X$ is~$-\zeta^{-1}\,d\zeta$. Pairing this form with~$X$ gives the constant function~$-1$. The leaf~$\widehat{\mathcal{L}}$ (but not~$\mathcal{L}$!) is a covering of the orbifold~$B^*$ and gives it, through the coordinate~$\zeta$, an affine structure which, on its turn, can be lifted to any leaf of~$\mathcal{F}$. In a leaf, the difference of this affine structure with the one induced by~$X$ is a one-form, which may be contracted with~$X$. This produces a function in~$S$ whose restriction to~$\widehat{\mathcal{L}}$ is identically~$-1$, and that is thus the constant function~$-1$. If~$L$ is a leaf of~$\mathcal{F}$ and~$\zeta$ is an affine coordinate induced by the one in~$B^*$ in which the restriction of~$X$ reads~$f(\zeta)\indel{\zeta}$, the difference of the affine structure induced by~$X$ and the one induced by~$\zeta$  is~$-f'/f\,d\zeta$, and pairing this form with the vector field gives the function~$-f'$. Since~$-f'\equiv -1$, $f(\zeta)=\zeta+c$ for some~$c\in\mathbf{C}$. Thus, for every leaf of~$\mathcal{F}$ where the vector field is holomorphic and nonzero at a generic point, in the affine coordinate inherited from~$B^*$, the vector field~$X$ is linear. In particular, 
\begin{itemize} \item in the neighborhood of~$p$ there is a component  of the curve of zeros that intersects transversely every leaf; and
	\item no leaf intersects the curve of poles (which is thus invariant by~$\mathcal{F}$).
\end{itemize}
From the first fact, since~$\mathcal{L}$ intersects only once the curve of zeros, it cannot self-accumulate and, since it intersects infinitely many times each fiber, the monodromy of the equation is discrete. From the second one, since there must be at least one component of the curve of poles which is not contained in a fiber (for otherwise the fibration would be preserved) and since this component is invariant, the monodromy group of the equation has a finite orbit. We thus conclude that~$\Gamma$ is discrete and elementary, a case for which we have already established our result in section~\ref{con:discele}.

We have proved the Main Theorem in all cases where the induced foliation is a Riccati one, for we have proved it
\begin{itemize}
	\item when~$\Gamma$ is finite, for in this case the intersection of~$\mathcal{L}$ with a generic fiber is finite (proposition~\ref{notfinite});
	\item when $\Gamma$ is infinite, discrete  and elementary (section~\ref{con:discele});
	\item when $\Gamma$ is infinite, discrete and nonelementary, for the limit set of the action of~$\Gamma$ on a generic fiber~$H$, which is nonempty, closed and invariant, and minimal with respect to these three properties is  contained in~$\mathfrak{U}\cap H$, and since the limit set is uncountable, the latter is  uncountable as well;
	\item when~$\Gamma$ is not discrete, for~$\mathfrak{U}\cap H$, being the union of orbits of~$\overline{\Gamma}$, is either finite  or uncountable. 
\end{itemize}

\subsection{Turbulent}\label{sec:turb} The proof will follow the lines of that of the Riccati case, from which we borrow the notation. Let~$\Pi:S\to B$ be an elliptic fibration, $X$ a meromorphic vector field on~$S$ inducing a turbulent foliation adapted to~$\Pi$ without first integrals.

Let~$\Sigma$ be a component of the curve of zeros that is not contained in a fiber (it intersects all the fibers). If~$\Sigma$ is invariant by~$\mathcal{F}$, the intersection of~$\Sigma$ with a generic fiber~$H$ gives a finite orbit for the action of the monodromy of the foliation on~$H$. Since a group of automorphisms of an elliptic curve with one finite orbit is finite, the monodromy must be finite. Further, all special fibers must be dicritical, for nondicritical fibers do not have separatrices. In this case there is a first integral for~$\mathcal{F}$. If~$\Sigma$ is not invariant, in the neighborhood of the nondicritical special fibers, every leaf will intersect~$\Sigma$ infinitely many times. Hence, all the special fibers are dicritical.  We may  suppose that the monodromy is infinite, for otherwise there would be a first integral for~$\mathcal{F}$.  In particular, all the orbits self-accumulate, and the closure of any orbit contains uncountably many others. If~$p\in\widehat{\mathcal{L}}\setminus\mathcal{L}$, $p$ lies at the intersection of~$\widehat{\mathcal{L}}$ with the curve of zeros. But, by the self-accumulation of~$\widehat{\mathcal{L}}$, this would imply that the orbit intersects infinitely many times the curve of zeros. Thus, $\mathcal{L}=\widehat{\mathcal{L}}$. Let us go again through the arguments of the Riccati case in section~\ref{par:firstcase}. Consider~$B$ with its orbifold structure. Up to replacing~$\mathcal{L}$ by another orbit, the projection~$\Pi|_{\mathcal{L}}$ is a covering of~$B$ (in the orbifold sense). Thus, $B$ carries an affine structure making the restriction of~$\Pi$ to every leaf an affine map (which is regular at a generic point of~$B$). But if~$H$ is a generic fiber, the leaf through~$\Sigma\cap H$ has a singularity for its affine structure at~$\Sigma\cap H$, and this contradicts the previous assertion.  Thus, all components of the curve of zeros are contained in fibers and all the components of the curve of poles must do so as well, and the fibration is preserved.

We thus finish the proof of our Main Theorem in the case where the surface is algebraic and the solution is entire.  

\section{The nonalgebraic case}\label{sec:nonalg} The  case remaining to be considered is the one where~$S$ is K\"ahler but not algebraic. Brunella classified holomorphic foliations in these surfaces~\cite[section~10]{brunella-pisa}. We denote by~$a(S)$ the algebraic dimension of the compact complex surface~$S$.

\begin{theorem}[Brunella] Let~$S$ be a minimal nonprojective K\"ahler compact complex surface and~$\mathcal{F}$ a holomorphic foliation on~$S$. Either:
	\begin{itemize}
		\item $S$ is an elliptic fibration with~$a(S)=1$, and~$\mathcal{F}$ either coincides with the fibration or is turbulent with respect to it.
		\item $S$ is an Abelian surface with $a(S)=0$, and~$\mathcal{F}$ is induced by a constant vector field.
		\item $S$ is a $K3$ surface with $a(S)=0$, there is an Abelian surface~$A$ and a covering~$\Pi:r\to S$ such that $\mathcal{F}$ is induced through~$r$ by a constant vector field on~$A$. 
	\end{itemize}
\end{theorem}

Let us prove the Main Theorem in the nonalgebraic, K\"ahler setting by using this result. We use the notations of the statement of the Main Theorem, with~$S$ a non-K\"ahler surface (that we may suppose minimal), and  denote by~$\mathcal{F}$ the foliation induced by~$X$. Let us consider the possibilities given by the above result. In the first case  all the algebraic curves are tangent to the fibration and so are the of poles of~$X$, and thus the fibration given by the algebraic reduction is preserved. In the second case, the ratio between~$X$ and a constant vector field inducing~$\mathcal{F}$ is a meromorphic function on~$S$, which must be constant. The same argument applies in the third case to the ratio of the pull-back of~$X$ and of the constant vector field: $X$ is a holomorphic vector field on~$S$. 

This completes the proof of the Main Theorem.

\section{Final comments and examples}\label{sec:excom}

For Riccati equations, having a first integral implies that the monodromy group is finite. The reciprocal is not true. The following example, which does not have a single algebraic solution, is attributed to Wittich~\cite[section~4.1]{hille}. The Riccati differential equation
\[y'=y^2+\frac{t^3+2}{t(t^3-1)}y+\frac{(t^3-1)^2}{t^4}\]
has the transcendental solutions
\[y(t)=\left(t-\frac{1}{t^2}\right)\tan\left(\frac{t^2}{2}+\frac{1}{t}+c\right).\]
Since all the solutions are single-valued, the monodromy is trivial. However, the solutions have an essential singularity at~$t=0$ and are not algebraic (in particular, the equation cannot have a first integral).\\

The relation between the univalent maximal solutions and the fixed points of the monodromy representation in both Riccati and turbulent foliations is behind corollary~\ref{treshold}.

\begin{proof}[Proof of Corollary~\ref{treshold}] Let us go through the possibilities in the conclusion of the Main Theorem. If the vector field~$X$ is birationaly equivalent to a holomorphic vector field, all its maximal solutions are univalent. Let us address the Riccati case, the one where there is a rational fibration~$\Pi:S\to B$ preserved by~$X$. Let~$B_0\subset B$ be the image of the nonspecial fibers. Let~$\rho:\pi_1(B_0)\to\mathrm{PSL}(2,\mathbf{C})$ be the monodromy of the Riccati foliation. Let~$Y=\Pi_*X$ and let~$\mu:\pi_1(B_0)\to\mathbf{C}$ be the monodromy of the translation structure induced by~$Y|_{B_0}$. The univalent maximal solutions of~$X$ correspond to the global fixed points of the restriction of~$\rho$ to~$\mathrm{ker}(\mu)\subset\pi_1(B_0)$. If~$X$ has three univalent maximal solutions with different images, this representation has three fixed points, and  is thus trivial. Hence, all the points are fixed, and all the solutions are single-valued. We proceed similarly for the turbulent case: the univalent maximal solutions are the fixed points of a representation analogous to the one previously described, replacing the rational fibration by an elliptic one. The corollary follows from the fact that an automorphism of an elliptic curve having five fixed points is the identity. \end{proof}

The elliptic involution of an elliptic curve has four fixed points. We have the following example: 

\begin{example}\label{four} Consider, in~$\mathbf{C}\times \mathbf{C}$, the meromorphic vector field
	\begin{equation}\label{above}X=\frac{1}{2x}\del{x}-\frac{1}{x^3}\del{y}.\end{equation}
	It is invariant by translations in~$y$ and by the order two-mapping~$\sigma$
	\begin{equation}\label{exinvolution} (x,y)\stackrel{\sigma}{\rightarrow} (-x,-y). \end{equation}
	It has the first integral $y-2/x$. Each one of its level curves may be parametrized by~$x$ and is thus biholomorphic to~$\mathbf{C}^*$. In this global coordinate the restriction of the vector field is $(2x)^{-1}\indel{x}$. Let~$\Lambda\subset\mathbf{C}$ be a lattice. Since the vector field~(\ref{above}) is invariant by translations in~$y$, it induces a vector field on~$M=\mathbf{C}\times (\mathbf {C}/\Lambda)$, which is an elliptic fibration with respect to the projection onto the first factor. The foliation  induced by~$X$ is a turbulent one with respect to this fibration. The elliptic fibration is equivariant with respect to the involution induced by~(\ref{exinvolution}) on~$M$ and by~$x\mapsto -x$ on~$\mathbf{C}$. The quotient~$S$ of~$M$ under~$\sigma$ is, after resolution of the singularities produced by the four fixed points, an elliptic fibration~$\Pi:S\to \mathbf{C}$. It is still endowed with a vector field, that projects onto~$\indel{z}$ under the mapping~$(x,y)\mapsto x^2$. For the four values of~$c$ in~$\frac{1}{2}\Lambda$, the leaf~$y=c+2/x$ of~(\ref{above}) is fixed under the involution and the projection~$\Pi$ in restriction to this leaf is injective (and corresponds thus to a univalent maximal solution). In the other leaves the solution is multivalued and has two determinations. This produces a surface having a vector field with exactly four single-valued solutions.  The mapping $\Pi:S\to \mathbf{C}$ gives an elliptic fibration preserved by the meromorphic vector field, which induces a turbulent foliation. \end{example}

Some classical examples of algebraic differential equations on surfaces corresponding to the first possibility of the Main Theorem (giving holomorphic vector fields on surfaces) are given, for example, by the Chazy~VI equation $\phi'''= \phi\phi''+5(\phi')^2-\phi^2\phi'$ and the Chazy~IX equation
$\phi''' = 18(\phi'+\phi^2)(\phi'+3\phi^2)-6(\phi')^2$. They both have algebraic first integrals and their restriction to a generic level surface is birationally equivalent, respectively, to a linear vector field on~$\mathbf{P}^2$ or to a constant vector field on an Abelian surface (see~\cite{guillot-chazy} for details).

Let us end with an example related to the second possibility of the Main Theorem.

\begin{example}\label{ex:notriccati1} Consider the vector field~$X$ on~$\mathbf{C}\times\mathbf{P}^1$ given in~$\mathbf{C}\times\mathbf{C}$ by~$(z-w)\indel{z}$. It is invariant under the diagonal action of~$\mathrm{Aff}(\mathbf{C})$ on~$\mathbf{C}\times\mathbf{C}$. Let~$\Gamma\subset\mathrm{Aff}(\mathbf{C})$ be a crystallographic group. Let~$\nu$ be the smallest natural number such that~$\gamma^\nu$ is a translation for every~$\gamma\in\Gamma$. Let~$B$ be the quotient of~$\mathbf{C}$ under~$\Gamma$. The quotient of~$\mathbf{C}\times\mathbf{P}^1$ under the diagonal action of~$\Gamma$ is a surface~$S$ endowed with a rational fibration~$\Pi:S\to B$ with respect to which the foliation induced by~$X$ is a Riccati one. Notice that this fibration is not preserved by the flow of~$X$ on~$S$. However, there is an elliptic fibration~$\Xi:S\to\mathbf{P}^1$ induced by~$(z,w)\mapsto (z-w)^\nu$ that maps~$X$ to the vector field~$\nu\xi\indel{\xi}$. This elliptic fibration is preserved by the vector field. There are two special fibers: one dicritical and one invariant.
\end{example}


\providecommand{\bysame}{\leavevmode\hbox to3em{\hrulefill}\thinspace}
 \providecommand{\href}[2]{#2}

\end{document}